\documentclass[a4paper,11pt,leqno]{article}

\usepackage[latin1]{inputenc}
\usepackage[T1]{fontenc} 
\usepackage[english]{babel}
\usepackage{verbatim}
\usepackage{calligra}
\usepackage{enumerate}
\usepackage{dsfont}
\usepackage{amsfonts}
\usepackage{amsmath}
\usepackage{amsthm}
\usepackage{amssymb}
\usepackage{fancyhdr}
\usepackage{mathtools}
\usepackage{mathrsfs}
\usepackage{color}
\usepackage{graphicx}
\usepackage{bm}

\theoremstyle{definition}
\newtheorem{defin}{Definition}[section]

\newtheorem{rem}[defin]{Remark}

\theoremstyle{plane}
\newtheorem{thm}[defin]{Theorem}
\newtheorem{prop}[defin]{Proposition}
\newtheorem{coroll}[defin]{Corollary}
\newtheorem{lemma}[defin]{Lemma}

\newcommand{\mbb}{\mathbb}

\newcommand{\mc}{\mathcal}
\newcommand{\mf}{\mathfrak}
\newcommand{\mds}{\mathds}
\newcommand{\veps}{\varepsilon}
\newcommand{\what}{\widehat}
\newcommand{\wtilde}{\widetilde}
\newcommand{\vphi}{\varphi}
\newcommand{\oline}{\overline}

\newcommand{\ra}{\rightarrow}

\newcommand{\hra}{\hookrightarrow}

\newcommand{\g}{\gamma}
\renewcommand{\k}{\kappa}
\newcommand{\s}{\sigma}
\renewcommand{\t}{\tau}
\newcommand{\z}{\zeta}
\newcommand{\lam}{\lambda}
\newcommand{\de}{\delta}
\newcommand{\lan}{\langle}
\newcommand{\ran}{\rangle}

\newcommand{\R}{\mathbb{R}}
\newcommand{\Q}{\mathbb{Q}}

\newcommand{\N}{\mathbb{N}}
\newcommand{\Z}{\mathbb{Z}}
\newcommand{\T}{\mathbb{T}}
\renewcommand{\P}{\mathbb{P}}

\renewcommand{\div}{{\rm div}\,}
\newcommand{\divh}{{\rm div}_h}
\newcommand{\curlh}{{\rm curl}_h}

\newcommand{\curl}{{\rm curl}\,}

\newcommand{\Id}{{\rm Id}\,}

\newcommand{\trho}{\widetilde{\rho}}
\newcommand{\ess}{{\rm ess}}
\newcommand{\res}{{\rm res}}

\allowdisplaybreaks

\def\d{\partial}
\def\div{{\rm div}\,}

\title{\textsc{\Large{\textbf{Incompressible and fast rotation limit for barotropic Navier-Stokes equations at large Mach numbers}}}}

\author{\normalsize\textsl{Francesco Fanelli 
} \vspace{.5cm} \\
\footnotesize{\textsc{Universit\'e de Lyon, Universit\'e Claude Bernard Lyon 1}} \\
{\footnotesize \it Institut Camille Jordan -- UMR 5208} \vspace{.1cm} \\
\footnotesize{\ttfamily{fanelli@math.univ-lyon1.fr}}
}

\vspace{.2cm}
\date\today

\begin{document}

\maketitle

\subsubsection*{Abstract}
{\footnotesize In the present paper we study the incompressible and fast rotation limit for the barotropic Navier-Stokes equations with Coriolis force,
in the case when the Mach number $\rm Ma$ is large with respect to the Rossby number $\rm Ro$: namely, we focus on the regime ${\rm Ro}\ll{\rm Ma}$. For this, we
follow a recent approach by Danchin and Mucha in \cite{D-M} and take also a large bulk viscosity coefficient.
We prove that the limit dynamics is described by an incompressible Navier-Stokes type equation, recasted in the vorticy formulation, where however an additional unknown,
linked to density oscillations around a fixed constant reference state, comes into play.
The proof of the convergence is based on a compensated compactness argument and on the derivation of sharp decay estimates for solutions to a heat equation with fast diffusion in time.}

\paragraph*{2010 Mathematics Subject Classification:}{\small 35Q35 
(primary);
35B25, 
76U05, 
35Q86, 
35B40, 
76M45 
(secondary).}

\paragraph*{Keywords:}{\small barotropic Navier-Stokes; incompressible limit; large Mach number; fast rotation; low Rossby number; large bulk viscosity; singular perturbation problem; multi-scale
limit.}

\section{Introduction} \label{s:intro}

In this paper we are interested in studying the dynamics of viscous barotropic fluids which undergo the action of a strong Coriolis force. The main application we have in mind is
to describe flows in the atmosphere; then, two are the main features we want to retain (see e.g. \cite{CR}, \cite{Ped}): on the one hand, almost incompressibility of the flow, on the other hand
the importance of the Earth rotation on the fluid motion.

Before going further, let us present the system of equations which are central to our study.

\subsection{Presentation of the model} \label{ss:i-model}

Let the scalar function $\rho\geq0$ denote the density of the fluid and $u\in\R^3$ its velocity field: forgetting about temperature variations, the model we are going to consider is given by
a \emph{$3$-D compressible Navier-Stokes system with Coriolis force}. In its non-dimensional form, this sytem can be written as follows (see e.g. \cite{F-N}, \cite{Klein} and \cite{Klein_2010}):
\begin{equation} \label{eq:resc-NSC}
\left\{\begin{array}{l}
        \d_t\rho\,+\,\div\big(\rho\,u\big)\,=\,0 \\[1ex]
       \d_t\big(\rho\,u\big)\,+\,\div\big(\rho\,u\otimes u\big)\,+\,\dfrac{1}{\rm Ma^{2}}\nabla P(\rho)\,+\,\dfrac{1}{\rm Ro}\,e^3\times\rho\,u\,-\,\dfrac{1}{\rm Re}\,\div\mbb S(\nabla u)\,=\,0\,.
       \end{array}
\right.
\end{equation}
We set the previous system in the domain $\Omega\,=\,\R^2\times\,]0,1[\,$ (considering the case of $\T^2\times\,]0,1[\,$ would however require slight adaptations),
supplemented with \emph{complete slip} boundary conditions. Such an hypothesis is
a true simplification, since it allows to avoid appearence of boundary layers (the so-called Ekman boundary layers) when considering the fast rotation limit.

The former equation in system \eqref{eq:resc-NSC} is called continuity (or mass) equation, the latter one is referred to as momentum equation.
The scalar function $P\,=\,P(\rho)$ appearing in the momentum equation represents the pressure of the fluid; it is supposed to be a smooth function of the density only.
The term $e^3\times\rho u$, where $e^3\,=\,(0,0,1)$ is the unit vector directed along the
vertical direction and the symbol $\times$ denotes the usual external product of vectors in $\R^3$, takes into account the action of the Coriolis force on the flow,
due to fast rotation of the Earth. Such a form of the Coriolis term is an approximation which is physically well justified at mid-latitudes
(see for instance \cite{C-D-G-G}, \cite{CR} and \cite{Ped}).
Finally, the term $\mbb S(\nabla u)$ is the viscous stress tensor, and it is assumed to satisfy
the Newton's rheological law (see e.g. \cite{F-N})
$$
\mbb{S}(\nabla u)\,=\,\mu\,\left(\nabla u\,+\,^t\nabla u\,-\,\frac{2}{3}\,\div u\,\Id\right)\,+\,\lambda\,\div u\,\Id\,,
$$
where the coefficients $\mu$ and $\lambda$ are called respectively the \emph{shear viscosity} and \emph{bulk viscosity} coefficients. Throughout this paper, we assume that the values of
$\mu$ and $\lambda$ do not depend on the density (nor on the temperature, of course), and that they are strictly positive constants (although such a requirement is not really necessary
for the well-posedness theory of the previous system, for which we refer to \cite{Lions_2} and \cite{F-N-P}).

In the momentum equation, the three adimensional parameters
${\rm Ma}$, ${\rm Ro}$ and ${\rm Re}$ come into play: they represent respectively the Mach, Rossby and Reynolds numbers. The \emph{Mach number} is connected with incompressibility: the lower its value is,
the most the flow tends to behave like an incompressible fluid. The \emph{Rossby number} measures the importance of Earth rotation effects on the dynamics of the fluid: having a low Rossby number
means that the Coriolis force has a predominant effect on the dynamics and then, according to Taylor-Proudman theorem (see e.g. \cite{CR} and \cite{Ped}), the flow tends to be planar and horizontal.
Finally, the \emph{Reynolds number} represents the ratio of inertial forces to viscous forces for given flow conditions; it measures somehow the turbulent behaviour of a fluid:
having a large Reynolds number means that effects of viscosity are negligible and the fluid tends to be turbulent.

Having in mind applications to the study of geophysical flows, 
it is natural for us to consider system \eqref{eq:resc-NSC} in a low Rossby number regime.
Our main goal here is to perform the fast rotation limit in the case of \emph{large Mach numbers}. In order to explain better this claim, let us give an overview of
previous results on similar problems.

\subsection{Previous results, motivations} \label{ss:i-motiv}

The mathematical theory of fluids in fast rotation has now a quite long history, which goes back to the pioneering works of Babin, Mahalov and Nikolaenko
\cite{B-M-N_1996}-\cite{B-M-N_1997}-\cite{B-M-N_1999} concerning the $3$-D incompressible Navier-Stokes equations. We refer to the book \cite{C-D-G-G} for
a complete treatment of the problem for that model, and for further references.

Reviewing the whole literature on the subject goes beyond the scopes of the present introduction. For this reason, we prefer to give a short overview of it, focusing mainly
on the results which are relevant for our study.

The fast rotation limit for fluids presenting density variations is a much more recent topic. In the compressible case, preliminary results were obtained in \cite{B-D_2003} (for the $2$-D case)
and \cite{B-D-GV_2004}, but for well-prepared data only. Dealing with general ill-prepared data in a $3$-D domain was reached for the first time (to the best of our knowledge)
in paper \cite{F-G-N} by Feireisl, Gallagher and Novotn\'y. Afterwards, more general multi-scale limits (still in $3$-D, for ill-prepared initial data) have been considered: for instance,
in \cite{F-G-GV-N} the contribution of the centrifugal force was added to the system, in \cite{F-N_AMPA}-\cite{F-N_CPDE} the interaction with the gravitational force was studied
in a regime of low stratification (see \cite{F-Lu-N} for the case of strong stratification, for well-prepared initial data only).
In this context, let us mention also the study of \cite{G-SR_Mem} concerning the so-called betaplane model (see also the review article \cite{G_2008}), paper \cite{K-M-N},
which is the first one dealing with heat conducting fluids, and works \cite{F_MA}-\cite{F_JMFM}, concerning a Navier-Stokes-Korteweg model with Coriolis force (the results therein
somehow generalise \cite{B-D_2003} under the point of view of the space dimension, the multiple regimes one may consider and the ill-preparation of the initial data).

For the sake of completeness, we point out that, on the side of density-dependent incompressible fluids, the only available study seem to be the one of \cite{F-G_RMI}, which however holds
in two space dimensions.

Let us now come back to the case of viscous compressible flows, which is the relevant framework for us.
The common point of all the previous references was to combine the fast rotation limit (i.e. low Rossby number) together with the incompressible limit (low Mach number). Notice that, as mentioned
above, 
such an investigation is well-justified from the physical viewpoint, for instance if one is interested in describing flows in the atmosphere.
Let us be more precise: given a small parameter $\veps\in\,]0,1]$ and a real number $\alpha\geq0$, in \eqref{eq:resc-NSC} we set
\begin{equation} \label{eq:scalings}
{\rm Ma}\,=\,\veps^{\alpha}\qquad\qquad\mbox{ and }\qquad\qquad {\rm Ro}\,=\,\veps\,.
\end{equation}
All the previous works focused on either the regime $\alpha$ large (due to technical restrictions, $\alpha\geq 10$ in \cite{F-G-GV-N}, $\alpha>2$ in \cite{F-N_CPDE}), or on the regime
$\alpha=1$ (see \cite{F-G-N}, \cite{F-G-GV-N}). The former framework gives rise to a multi-scale problem, where the incompressibility effect is predominant; the latter
is the case of isotropic scaling, and allows one to recover the so-called \emph{quasi-geostrophic balance}, where weak compressibility and fast rotation act at the same order,
and they keep in balance in the limit when $\veps\ra0^+$
(then the asymptotic dynamics is described by a quasi-geostrophic type equation).

We remark that, up to now, the parameter $\rm Re$ has played no special role in the study, and it can be taken equal to $1$ in the previous discussion.
Nonetheless, it is remarkable that in \cite{F-N_AMPA}-\cite{F-N_CPDE} (see also \cite{K-M-N}) the authors are able to perform the limit even in the case of large Reynolds
numbers (namely, ${\rm Re}\,\sim\,\veps^{-k}$, for some $k>0$) by resorting to the relative entropy method; of course, they identify an inviscid equation as the target dynamics.

The main motivation of this paper is to understand what happens in the regimes of large Mach numbers, 
in the sense that ${\rm Ro}\ll{\rm Ma}$. 
More precisely, we want to consider the situation when one takes $0\leq\alpha<1$ in \eqref{eq:scalings}, which have been left open so far.

\subsection{Contents of the paper and overview of the results} \label{ss:i-results}

After noticing that
$$
\div\mbb S(\nabla u)\,=\,\mu\,\Delta u\,+\,\left(\frac{1}{3}\,\mu\,+\,\lambda\right)\,\nabla\div u
$$
and sightly changing the notation for the viscosity coefficients, we can rewrite system \eqref{eq:resc-NSC} in the form
\begin{equation} \label{eq:intro-NSC}
\left\{\begin{array}{l}
        \d_t\rho\,+\,\div\big(\rho\,u\big)\,=\,0 \\[1ex]
       \d_t\big(\rho\,u\big)\,+\,\div\big(\rho\,u\otimes u\big)\,+\,\dfrac{1}{\veps^{2\alpha}}\nabla P(\rho)\,+\,\dfrac{1}{\veps}\,e^3\times\rho\,u\,-\,
       \mu\,\Delta u\,-\,\lambda\,\nabla\div u\,=\,0\,.
       \end{array}
\right.
\end{equation}
Here below, for simplicity we will refer to $\mu$ as the shear viscosity and to $\lambda$ as the bulk viscosity, although (in view of what we have said above)
such a terminology is a bit improper.

As explained before, we are interested in the regimes when $\alpha\in[0,1[\,$, namely the fast rotation is the predominant effect in the dynamics. Indeed, the cases when $\alpha\geq 1$
have already been considered in previous works.
Nonetheless, an easy inspection of the momentum equation in \eqref{eq:intro-NSC} reveals that the limit velocity field is trivial, namely $u\equiv0$, if the 
strong Coriolis force is not compensated by a gradient.

In order to unlock such an \textsl{impasse}, we decide to adopt the approach of the recent paper \cite{D-M} by Danchin and Mucha.
There, the authors considered the problem of performing the incompressible limit for the barotropic compressible Navier-Stokes system (without Coriolis force),
by letting $\lambda\ra+\infty$ in \eqref{eq:intro-NSC}, rather than taking $\alpha>0$ (i.e. a small Mach number). By following this strategy, the authors in \cite{D-M}
are able to prove global existence in critical spaces for \eqref{eq:intro-NSC}, with $\alpha=0$ and without Coriolis force, both in space dimension $d=2$ and $3$,
by exploiting the global well-posedness of the limit problem (which is always true when $d=2$, and assumed \textsl{a priori} when $d=3$).

Inspired by \cite{D-M}, in addition to the previous scalings, we consider in \eqref{eq:intro-NSC} a large bulk viscosity $\lambda=\lam(\veps)\ra+\infty$
for $\veps\ra0^+$. More precisely, we take $\lam=1/\veps^{2\beta}$. Once again, it is easy to see (check also Remark \ref{r:beta} below) that, if $0\leq \beta<1$, the limit is still trivial.
The reason is that the effect of the gradient is not strong enough to compensate the fast rotation, which is still the main feature and then tends to kill off the other processes of the dynamics.

Therefore, we finally fix the choices $0\leq\alpha<1$ and $\lambda=\veps^{-2\beta}$ with $\beta\geq1$, in system \eqref{eq:intro-NSC}. We want to study the asymptotic limit
of this system when $\veps\ra0^+$ in the context of weak solutions. Notice that this is a singular limit problem, where multiple scales act at the sime time, but with different strengths,
on the system. One may object that, having a large bulk viscosity which implies incompressibility of the limit flow, the presence of a small Mach number is useless, and then
one should rather fix $\alpha=0$. Still, we are able to treat the endpoint case $\alpha=0$ only when the space dimension is $d=2$: we will come back to this issue in a while.

To begin with, let us detail our framework. First of all, we will consider \emph{ill prepared initial data}, where however the initial density perturbations around a constant state $\oline\rho$,
say $\oline\rho=1$, are of size $\veps$ (i.e. the same size as the Rossby number) rather than $\veps^\alpha$ (the size of the Mach number) as one might expect.
At first glance, this assumption may appear useless, since in any case classical energy estimates (the only bounds we will use for our family of weak solutions) allow to show that,
at any later time, one only has $\rho_\veps(t)-1\,=\,O(\veps^\alpha)$, in a suitable topology. Nevertheless, thanks to the additional smallness of the initial density perturbations,
by resorting to an argument used in \cite{F-G_RMI} for the incompressible case, we will be able to show uniform bounds (in spaces of very low regularity with respect to the space variable)
on the vertical means of the  quantities $\s_\veps(t)\,:=\,\big(\rho_\veps(t)-1\big)/\veps$.
Such a remarkable property is unexpected from classical energy estimates: in fact, it deeply relies on the structure of the wave system which governs the propagation of fast time oscillations
(due to the ill-preparation of the initial data), which we will call \emph{acoustic-Poincar\'e waves}.

Remark that the previous argument is particularly important in the endpoint case $\alpha=0$, since at first glance (based on energy estimates) one disposes of no smallness at all on the
density perturbations $\rho_\veps(t)-1$. 
Nevertheless, as pointed out above, in this way one gains smallness only on the vertical means of the quantites $\s_\veps$, whereas
a global smallness (even very rough, but quantified in terms of powers of $\veps$) of the quantities $\rho_\veps(t)-1$ is still needed in order to pass to the limit
in the weak formulation of the equations \eqref{eq:intro-NSC}. This is thereason why, when $\alpha=0$, we have to restrict our attention to $2$-dimensional flows: then the uniform bounds
are valid on the whole quantity $\sigma_\veps$ (there is no more need to take vertical averages), and we are able to make our argument work.

Let us resume the overview of our strategy, coming back to the general $3$-D case (but the same applies also in the $2$-D case).
The bounds on $\sigma_\veps$ having been established, the rest of the proof is based on a compensated compactness argument, which allows to prove convergence of the most non-linear term, i.e.
the convective term in the second equation of \eqref{eq:intro-NSC}. Such a technique goes back to the pioneering work \cite{L-M} by P.-L. Lions and Masmoudi,
where the authors dealt with the incompressible limit of the compressible Navier-Stokes equations; it was later adapted by Gallagher and Saint-Raymond in \cite{G-SR_2006}
to the context of fast rotating fluids, and then broadly exploited in similar studies (see e.g. \cite{F-G-GV-N}, \cite{F_JMFM}, \cite{F-G_RMI}).

However, the compensated compactness argument allows to say that the convergence of the convective term reduces, up to small remainders which vanish in the limit, to the convergence of a bilinear term
$\mc B\big(\lan\eta^3_\veps\ran,\lan V_\veps\ran)$, which depends on the vertical avergages (this is the meaning of the notation $\lan\,\cdot\,\ran$)
of both the momentum $V_\veps\,=\,\rho_\veps u_\veps$ and the vertical component of its vorticity $\eta^3_\veps\,=\,\d_1V_\veps^2-\d_2V_\veps^1$. Therefore, in order to pass to the limit
in $\mc B$, we need compactness in time for one of the previous two quantities. Notice that taking the $\curl$ of the momentum equation in \eqref{eq:intro-NSC} yields an equation for $\eta^3_\veps$:
\begin{equation} \label{eq:intro-vort}
\veps\,\d_t\eta_\veps^3\,+\,\divh V^h_\veps\,=\,\veps\,F_\veps\,,
\end{equation}
where the notation $F_\veps$ encodes terms which are uniformly bounded in suitable spaces and we have set $\divh V^h\,=\,\d_1V^1\,+\,\d_2V^2$.
The problem is that the previous relation entails fast oscillations in time
for the vorticity $\eta^3_\veps$, unless we are able to show that also $\div V_\veps$ is small, and more precisely of order $O(\veps)$ in suitable norms. Notice that we are not too far from that
property, if one thinks that $\div V_\veps\,\sim\,\div u_\veps$, and the latter is of order $O(\veps^\beta)$, with $\beta\geq1$; but the difference between those two quantities is
only of order $O(\veps^\alpha)$ (even when $d=2$ and $\alpha=0$, because the regularity of $\s_\veps$ is too rough to give sense to the product $\s_\veps\,u_\veps$).
We then need additional smallness on the quantity $\div V_\veps$: notice that such a smallness cannot really come from the wave system, since acoustic-Poincar\'e waves
travel at speed of the Mach number, hence proportional to $\veps^\alpha$. For this reason, dispersive estimates used in e.g. \cite{F-G-GV-N}, \cite{F-N_CPDE} seem to be
out of use in our context. The fundamental remark is rather that the momentum equation in \eqref{eq:intro-NSC} hides a heat-like equation for the potential part
$\nabla\Phi_\veps$ of $V_\veps$, with fast diffusion in time:
\begin{equation} \label{eq:intro-Phi}
\d_t\Phi_\veps\,-\,\frac{1}{\veps^{2\beta}}\,\Delta\Phi_\veps\,=\,G_\veps\,.
\end{equation}
It is well-known that solutions to the heat equation decay in time (see e.g. \cite{Z}): we have to show the exact counterpart for the previous equation, where the long-time behaviour
is instead replaced by the asymptotic behaviour with respect to $\veps$. The key point is to get the sharp decay with respect to the singular parameter $\veps$, since we want
to insert that control in \eqref{eq:intro-vort}. On the one hand, for doing so we lose integrability for times close to $0$, so that we have to implement an additional approximation
procedure. On the other hand, the main difficulty comes from the fact that the forcing term $G_\veps$ in \eqref{eq:intro-Phi} is not uniformly bounded in $\veps$: then the idea is to differentiate
\eqref{eq:intro-Phi} as many times as one needs, steerred by the basic principle that the derivatives of the solution to the heat equation decay better than the solution itself.
In the end, we are able to gain smallness of $(-\Delta)^s\Phi_\veps$, for some $s\geq 1$ large enough. Inserting those bounds in \eqref{eq:intro-vort}, we get compactness
in time of higher order derivatives of $\eta^3_\veps$ and finally, interpolating this property with the uniform bounds for the vorticity, we get strong convergence in suitable
$L^r_T(L^p)$ spaces for $\lan\eta^3_\veps\ran$, which allows us to pass to the limit in the $\mc B$ term.

In the end, we can prove convergence in the vorticity formulation of the momentum equation in \eqref{eq:intro-NSC}: as it was already the case in \cite{F-G_RMI}, we find an \emph{underdetermined
limit equation}, which links both the limit vorticity $\omega\,=\,\d_1u^2-\d_2u^1$ (recall that the density tends to $1$ when $\veps\ra0^+$) and the limit $\s$ of the (vertical mean of the) density
variations $\s_\veps$. As already pointed out, the bounds on $\s_\veps$ are in too negative spaces in order to use the mass equation in \eqref{eq:intro-NSC} and deduce an equation for
$\s$ in the limit. This is the main result of the paper, which is contained in Theorem \ref{t:alpha} for the $3$-D case, in Theorem \ref{t:0} for $\alpha=0$ and $d=2$.
As a last comment, let us remark that, for the latter case $\alpha=0$ and $d=2$, we are also able to show a conditional convergence result (see Theorem \ref{t:2D-full}),
which allows to recover a full system in the limit, where the equations for both $\s$ and the limit velocity $u$ are well identified. However, such a result is based
on assuming \textsl{a priori} higher order bounds for the family $\big(u_\veps\big)_\veps$: on the one hand, those bounds seem to be hardly satisfied by finite energy weak solutions,
on the other hand higher order energy estimates seem to be not uniform in the singular parameter $\veps$. This is why our result is only conditional.

\medbreak
To conclude, let us give an overview of the paper. In the next section, we collect our assumptions and state our main results. Section \ref{s:tools} contains
some tools which are useful in our analysis: namely, some elements of Littlewood-Paley decomposition and paradifferential calculus, and also
the decay estimates for the linear equation \eqref{eq:intro-Phi}, which (as already remarked) play a key role in our proof. In Section \ref{s:singular} we study the singular part
of the equations, stating uniform bounds on our family of weak solutions and establishing constraints the limit-points have to satisfy.
Finally, in Section \ref{s:limit-a} we perform the limit in the weak formulation of system \eqref{eq:intro-NSC} when $0<\alpha<1$, in Section \ref{s:limit-0} when $\alpha=0$ and
$d=2$. As already mentioned, in the last part of Section 6 we will also state and prove our conditional result, where the limit dynamics for $\s$ and $u$ is completely characterised.

\paragraph{Notation.}
Let us introduce some notation here.

We will decompose $x\in\Omega\,:=\,\R^2\times\,]0,1[\,$ into $x=(x^h,x^3)$, with $x^h\in\R^2$ denoting its horizontal component. Analogously,
for a vector-field $v=(v^1,v^2,v^3)\in\R^3$ we set $v^h=(v^1,v^2)$, and we define the differential operators
$\nabla_h$ and $\div_{\!h}$ as the usual operators, but acting just with respect to $x^h$.
Finally, we define the operator $\nabla^\perp_h\,:=\,\bigl(-\d_2\,,\,\d_1\bigr)$ and, analogously, for a $2$-D vector-field $w$ we set $w^\perp\,=\,\big(-w^2,w^1\big)$.
For a $3$-D vector-field $v$, we will denote $\curl v\,=\,\nabla\times v$ its $\curl$, where the symbol $\times$ stands for the usual external product in $\R^3$; notice that
$\big(\curl v\big)^3\,=\,\curlh v^h\,=\,\d_1v^2\,-\,\d_2v^1$.
On the other hand, if $w$ is a $2$-D vector-field, we set $\curl w\,=\,\d_1w^2\,-\,\d_2w^1$.

Moreover, since we will reconduct ourselves to a periodic problem in the $x^3$-variable (see Remark \ref{r:period-bc} below),
we also introduce the following  decomposition: for a vector-field $X$, we write
\begin{equation} \label{dec:vert-av}
X(x)\,=\,\langle X\rangle(x^h)\,+\,\wtilde{X}(x)\,,\qquad\qquad\mbox{ where }\qquad
\langle X\rangle(x^h)\,:=\,\int_{\mbb{T}}X(x^h,x^3)\,dx^3\,.
\end{equation}
Notice that $\wtilde{X}$ has zero vertical average, and therefore we can write $\wtilde{X}(x)\,=\,\d_3\wtilde{Z}(x)$,
with $\wtilde{Z}$ having zero vertical average as well.
We also set $\wtilde{Z}\,=\,\mc{I}(\wtilde{X})\,=\,\d_3^{-1}\wtilde{X}$.

For convenience, for any $T>0$, $p\in[1,+\infty]$ and any Banach space~$X$ over $\Omega$, we will often use the notation
$L^p_T(X)\,:=\,L^p\bigl([0,T[\,;X(\Omega)\bigr)$. Moreover, we will use the symbols $\rightharpoonup$ and $\stackrel{*}{\rightharpoonup}$ to denote respectively
the weak and weak-$*$ convergences in the space $L^p_T(X)$.

We will denote by $\mc C^{0,\eta}$ the space of H\"older continuous functions of exponent $0<\eta<1$; in the endpoint case $\eta=1$, we will use the notation $W^{1,\infty}$.

\subsubsection*{Acknowledgements}
{\small
The work of the author has been partially supported by the LABEX MILYON (ANR-10-LABX-0070) of Universit\'e de Lyon, within the program ``Investissement d'Avenir''
(ANR-11-IDEX-0007),  and by the projects BORDS (ANR-16-CE40-0027-01) and SingFlows (ANR-18-CE40-0027), all operated by the French National Research Agency (ANR).

The author wants to express his gratitude to L. Brandolese for pointing out reference \cite{Z}, and to R. Danchin and I. Gallagher for interesting remarks on a preliminary version of the paper.
}

\section{Assumptions and results} \label{s:results}
Fix the domain
$$
\Omega\,:=\,\R^2\times\,]0,1[
$$
and take two real parameters
$$
0\leq\alpha<1\qquad\qquad\mbox{ and }\qquad\qquad \beta\geq1\,.
$$

We consider, on $\R_+\times\Omega$, the following rescaled $3$-D barotropic Navier-Stokes system with Coriolis force:
\begin{equation} \label{eq:sing-NSC}
\left\{\begin{array}{l}
        \d_t\rho\,+\,\div\big(\rho\,u\big)\,=\,0 \\[1ex]
       \d_t\big(\rho\,u\big)\,+\,\div\big(\rho\,u\otimes u\big)\,+\,\dfrac{1}{\veps^{2\alpha}}\nabla P(\rho)\,+\,\dfrac{1}{\veps}\,e^3\times\rho\,u\,-\,
       \mu\,\Delta u\,-\,\dfrac{1}{\veps^{2\beta}}\nabla\div u\,=\,0\,.
       \end{array}
\right.
\end{equation}
Throughout the paper, we assume that $P\in\mc C\big([0,+\infty[\big)\cap\mc C^2\big(]0,+\infty[\big)$, with
\begin{equation} \label{hyp:P}
P'(z)>0\qquad\mbox{ for all }\;z>0\,,\qquad\qquad\qquad \lim_{z\ra+\infty}\frac{P'(z)}{z^{\g-1}}\,=\,a\,>\,0\,,
\end{equation}
for some finite $\g>3/2$ in the case of a $3$-D domain, $\g>1$ in dimension $2$.
Let us immediately introduce the \emph{internal energy} function
$$
H(z)\,:=\,z\int^z_1\frac{P(z)}{z^2}\,dz\,.
$$
Notice that $H''(z)\,=\,P'(z)/z$. In addition, for any positive $\rho$ and $\trho$, let us define the \emph{relative entropy} functional
$$
E\big(\rho,\trho\big)\,:=\,H(\rho)\,-\,H(\trho)\,-\,H'(\trho)\,\big(\rho-\trho\big)\,.
$$

We supplement system \eqref{eq:sing-NSC} with complete slip boundary conditions: if we denote by $n$ the unitary outward normal to the boundary $\d\Omega$ of the domain (observe that
$\d\Omega=\{x^3=0\}\cup\{x^3=1\}$), we impose
\begin{equation} \label{eq:bc}
\left(u\cdot n\right)_{|\d\Omega}\,=\,u^3_{|\d\Omega}\,=\,0\,,\qquad
\bigl((Du)n\times n\bigr)_{|\d\Omega}\,=\,0\,,
\end{equation}
where $Du\,=\,\big(\nabla u\,+\,^t\nabla u\big)/2$ denotes the symmetric part of the tensor $\nabla u$.

\begin{rem} \label{r:period-bc}
As is well-known (see e.g. \cite{Ebin}), equations \eqref{eq:sing-NSC}, supplemented by complete slip boundary boundary conditions \eqref{eq:bc},
can be recasted as a periodic problem with respect to the vertical variable, in the new domain
$$
\Omega\,=\,\R^2\,\times\,\mbb{T}^1\,,\qquad\qquad\mbox{ with }\qquad\mbb{T}^1\,:=\,[-1,1]/\sim\,,
$$
where $\sim$ denotes the equivalence relation which identifies $-1$ and $1$. Indeed, the equations are invariant if we extend
$\rho$ and $u^h$ as even functions with respect to $x^3$, and $u^3$ as an odd function.

In what follows, we will always assume that such modifications have been performed on the initial data, and
that the respective solutions keep the same symmetry properties.
\end{rem}

We consider general \emph{ill-prepared} initial data. However, in order to perform the limit, it will be fundamental to derive further
compactness for the density function, following the approach proposed in \cite{F-G_RMI}. For this, we need the initial density to be close enough, at order $\veps$ (rather than $\veps^\alpha$,
i.e. at order dictated by the Mach number) to the limit state (say) $\oline\rho=1$.
Then, we assume the following conditions, for all $\veps\in\,]0,1]$:
\begin{enumerate}[(1)]
 \item $\rho_{0,\veps}\,=\,1\,+\,\veps\,r_{0,\veps}$, where $\big(r_{0,\veps}\big)_\veps\,\subset\,L^2(\Omega)\cap L^\infty(\Omega)$.
 \item the sequence $\big(u_{0,\veps}\big)_\veps$ is uniformly bounded in $L^2(\Omega)\cap L^\infty(\Omega)$.
\end{enumerate}
Up to the extraction of a suitable subsequence, which we do not relabel, we have
\begin{equation} \label{cv:in-data}
 r_{0,\veps}\,\stackrel{*}{\rightharpoonup}\,r_0\qquad\mbox{ and }\qquad  u_{0,\veps}\,\stackrel{*}{\rightharpoonup}\,u_0\qquad\qquad\mbox{ in }\qquad\qquad L^2(\Omega)\cap L^\infty(\Omega)\,,
\end{equation}
for suitable functions $r_0$ and $u_0$ belonging to that space.

For any $\veps\in\,]0,1]$ fixed, we supplement system \eqref{eq:sing-NSC} with the initial datum
$$
\rho_{|t=0}\,=\,\rho_{0,\veps}\qquad\qquad \mbox{ and }\qquad\qquad \big(\rho\,u\big)_{|t=0}\,=\,\rho_{0,\veps}\,u_{0,\veps}\,.
$$
We are interested in studying the asymptotic behaviour of system \eqref{eq:sing-NSC} in the framework of weak solutions. So, let us start by recalling their definition.
\begin{defin} \label{d:weak}
We say that $\bigl(\rho,u\bigr)$ is a \emph{weak solution} to equations \eqref{eq:sing-NSC}-\eqref{eq:bc}
in $[0,T[\,\times\Omega$ (for some $T>0$), related to the initial datum $(\rho_{0,\veps},u_{0,\veps})$ specified above, if the following conditions are verified:
\begin{enumerate}[(i)]
 \item $\rho\geq0$ almost everywhere, $\rho-1\,\in\,L^\infty\big([0,T[\,;L^2+L^\g(\Omega)\big)$ and the continuity equation 
 is satisfied in the weak sense: for all $\vphi\in\mc D\big([0,T[\,\times\Omega\big)$, one has
$$
-\int^T_0\int_\Omega\Big(\rho\,\d_t\vphi\,+\,\rho\,u\,\cdot\,\nabla\vphi\Big)\,dx\,dt\,=\,\int_\Omega\rho_{0,\veps}\,\vphi(0)\,dx\,;
$$
\item $\sqrt{\rho}u\,\in L^\infty\big([0,T[\,;L^2(\Omega)\big)$, $u\,\in\,L^2\big([0,T[\,;H^1(\Omega)\big)$, $P(\rho)\,\in\,L^1_{\rm loc}\big([0,T[\,\times\Omega\big)$
and the momentum equation 
is satisfied in the weak sense: for all $\psi\in\mc D\big([0,T[\,\times\Omega;\R^3\big)$, one has
\begin{align*}
&\hspace{-0.5cm}-\int^T_0\!\!\!\int_\Omega\left(\rho u\cdot\d_t\psi\,+\,\rho u\otimes u:\nabla\psi\,+\,\dfrac{1}{\veps^{2\alpha}}\,P(\rho)\,\div\psi\right)\,dx\,dt \\
&\quad +\,\int^T_0\!\!\!\int_\Omega\left(\dfrac{1}{\veps}\,e^3\times\rho\,u\cdot\psi\,+\,\mu\,\nabla u:\nabla\psi\,+\,\frac{1}{\veps^{2\beta}}\,\div u\;\div\psi\right)\,dx\,dt\,=\,
\int_\Omega\rho_{0,\veps}\,u_{0,\veps}\cdot\psi(0)\,dx\,;
\end{align*}
\item the following \emph{energy inequality} holds true for almost every $t\in[0,T[\,$:
\begin{align}
&\int_\Omega\left(\frac{1}{2}\,\rho(t)\,\big|u(t)\big|^2\,+\,\frac{1}{\veps^{2\alpha}}\,E\big(\rho(t),1\big)\right)\,dx \label{est:energy} \\
&+\,\int^t_0\!\int_\Omega\left(\mu\,\big|\nabla u\big|^2\,+\,\frac{1}{\veps^{2\beta}}\,\big|\div u\big|^2\right)\,dx\,dt\,\leq\,
\int_\Omega\left(\frac{1}{2}\,\rho_{0,\veps}\,\big|u_{0,\veps}\big|^2\,+\,\frac{1}{\veps^{2\alpha}}\,E\big(\rho_{0,\veps},1\big)\right)\,dx\,. \nonumber
\end{align}
\end{enumerate}
The solution is said to be \emph{global} if the previous conditions hold for all $T>0$.
\end{defin}

For any fixed value of the parameter $\veps\in\,]0,1]$, suppose an initial datum $\big(\rho_{0,\veps},u_{0,\veps}\big)$ is given, satisfying the hypotheses specified above.
The existence of a \emph{global in time} weak solution (in the sense of the previous definition) $\big(\rho_\veps,u_\veps\big)$ is guaranteed by the classical theory of P.-L. Lions \cite{Lions_2},
with the necessary modifications implemented in \cite{F-N-P} in order to handle the physically relevant range of adiabatic exponents $\g>3/2$ in dimension $3$,
$\g>1$ for a $2$-D space domain.

The main goal of the present paper is to characterise the limit dynamics of the sequence $\big(\rho_\veps,u_\veps\big)_\veps$ when $\veps\ra0^+$. We are interested in the regimes $0\leq\alpha<1$
(otherwiser the limit has already been performed) and $\beta\geq1$ (otherwise the limit is trivial, see the Introduction and Remark \ref{r:beta} below).

Our first main result concerns the $3$-D system, in the case when the Mach number is supposed to be small, i.e. $\alpha>0$.
\begin{thm}\label{t:alpha}
Let $0<\alpha<1$ and $\beta\geq1$. Take a sequence of initial data $\big(\rho_{0,\veps},u_{0,\veps}\big)_\veps$ verifying the assumptions stated above and the symmetry properties
of Remark \ref{r:period-bc}, and consider
a sequence $\big(\rho_\veps,u_\veps\big)_\veps$ of associated global finite energy weak solutions to system \eqref{eq:sing-NSC} in $\R_+\times\Omega$, in the sense of Definition \ref{d:weak}.
Let $u_0$ and $r_0$ be defined as in \eqref{cv:in-data}, and, for all $\veps>0$, set $r_\veps\,:=\,\big(\rho_\veps-1\big)/\veps^\alpha$ and $\sigma_\veps\,:=\,\big(\rho_\veps-1\big)/\veps$.

Then, $r_\veps\,\stackrel{*}{\rightharpoonup}\,0$ in the space $L^\infty\big(\R_+;\big(L^2+L^\gamma\big)(\Omega)\big)$. Moreover,
there exists a scalar distribution $\sigma\in L^2_{\rm loc}\big(\R_+;H^{-7/2-\delta}_{\rm loc}(\Omega)\big)$, for $\de>0$ arbitrarily small, and a vector-field
$u\in L^2_{\rm loc}\big(\R_+;H^1(\Omega)\big)$
of the form $u\,=\,\big(u^h,0\big)$, with $u^h\,=\,u^h(t,x^h)$ and $\divh u^h\,=\,0$, such that, up to the extraction of a subsequence, one has
$\lan\sigma_\veps\ran\,\stackrel{*}{\rightharpoonup}\,\sigma$ and $u_\veps\,\rightharpoonup\,u$ in the respective functional spaces. \\
Finally, after defining $\omega\,:=\,\curlh u^h$, the couple $\big(\omega,\sigma\big)$ satisfies (in the weak sense) the equation
\begin{equation} \label{eq:limit-a}
\d_t\big(\omega\,-\,\sigma\big)\,+\,u^h\cdot\nabla_h\omega\,-\,\mu\Delta_h\omega\,=\,0\,,
\end{equation}
supplemented with initial datum $\big(\omega-\sigma\big)_{|t=0}\,=\,\curlh\lan u^h_0\ran-\lan r_0\ran$.
\end{thm}


The natural question is then trying to reach the endpoint case $\alpha=0$: our second result is exactly about that framework. However, it turns out
that, for technical reasons, we are able to treat this case only when the fluid is supposed to be planar, and then the space dimension is equal to $2$.
Notice that, in this instance, system \eqref{eq:sing-NSC} becomes
\begin{equation} \label{eq:sing-NSC_2D}
\left\{\begin{array}{l}
        \d_t\rho\,+\,\div\big(\rho\,u\big)\,=\,0 \\[1ex]
       \d_t\big(\rho\,u\big)\,+\,\div\big(\rho\,u\otimes u\big)\,+\,\dfrac{1}{\veps^{2\alpha}}\nabla P(\rho)\,+\,\dfrac{1}{\veps}\,\rho\,u^\perp\,-\,
       \mu\,\Delta u\,-\,\dfrac{1}{\veps^{2\beta}}\nabla\div u\,=\,0\,,
       \end{array}
\right.
\end{equation}
set in $\R_+\times\R^2$. In the previous system, we have used the notation $v^\perp\,:=\,\big(-v^2,v^1\big)$. Of course, we dismiss the boundary conditions \eqref{eq:bc}.
In addition, in the $2$-D case, we omit to write the subscripts and superscripts ``h'', since now all the quantities are horizontal.

In this case, we can prove the following result.
\begin{thm}\label{t:0}
Let $\alpha=0$ and $\g>1$ in \eqref{eq:sing-NSC_2D}. 
Take a sequence of initial data $\big(\rho_{0,\veps},u_{0,\veps}\big)_\veps$ verifying the assumptions stated here above, and consider
a sequence $\big(\rho_\veps,u_\veps\big)_\veps$ of associated global finite energy weak solutions to system \eqref{eq:sing-NSC_2D} in $\R_+\times\R^2$, in the sense of Definition \ref{d:weak}.
Let $u_0$ and $r_0$ be defined as in \eqref{cv:in-data}, and, for all $\veps>0$, set $r_\veps\,:=\,\rho_\veps-1$ and $\sigma_\veps\,:=\,\big(\rho_\veps-1\big)/\veps$.

Then, $r_\veps$ convergese to $0$ in the weak-$*$ topology of $L^\infty\big(\R_+;\big(L^2+L^\gamma\big)(\R^2)\big)$.
Moreover, there exists a scalar distribution $\sigma\in L^2_{\rm loc}\big(\R_+;H^{-3-\delta}_{\rm loc}(\R^2)\big)$, for $\de>0$ arbitrarily small, and a vector-field
$u\in L^2_{\rm loc}\big(\R_+;H^1(\R^2)\big)$, with $\div u\,=\,0$, such that, up to the extraction of a subsequence, one has
$\sigma_\veps\,\stackrel{*}{\rightharpoonup}\,\sigma$ and $u_\veps\,\rightharpoonup\,u$ in the respective functional spaces. \\
Finally, after defining $\omega\,:=\,\curl u$, the couple $\big(\omega,\sigma\big)$ satisfies (in the weak sense) the equation
\begin{equation} \label{eq:limit-0}
\d_t\big(\omega\,-\,\sigma\big)\,+\,u\cdot\nabla\omega\,-\,\mu\Delta\omega\,=\,0\,,
\end{equation}
supplemented with initial datum $\big(\omega-\sigma\big)_{|t=0}\,=\,\curl u_0- r_0$.
\end{thm}

We conclude this part by remarking that we are able to treat also the case when an external force act on the system at the same order of the Mach number, and forces the limit
density profile $\trho$ to be non-constant (but still horizontal, i.e. $\trho\,=\,\trho(x^h)$, when the space dimension is equal to $3$),
finding in this way a linear equation describing the target dynamics. However, we refrain to deal with that problem here,
since we lack of physically relevant applications:
\begin{itemize}
 \item the gravitational force would imply rather $\trho\,=\,\trho(x^3)$, which is not a convenient setting for us;
 \item the centrifugal force (see e.g. \cite{F-G-GV-N}) would scale as the Rossby number, rather than the Mach number, then there would be no balance with the pressure in order to recover
$\trho\neq1$;
\item combining the presence of a capillarity term (see \cite{F_MA}-\cite{F_JMFM}) with a non-monotone pressure law would be convenient; but dealing with this case requires the use of the
BD-entropy structure of the system, which seems not to be compatible with the choice of constant viscosity coefficients.
\end{itemize}

\section{Tools} \label{s:tools}
The goal of this section is twofold: we start by recalling some tools from Fourier and harmonic analysis, which will be broadly used throughout our study. Then, we will present
decay estimates for parabolic equations with fast diffusion, which will play a key role in the convergence and seem to be of independent interest.

\subsection{Elements of Fourier and harmonic analysis} \label{app:LP}

We recall here the main ideas of Littlewood-Paley theory, which we will exploit in our analysis.
We refer e.g. to Chapter 2 of \cite{B-C-D} for details.
For simplicity of exposition, let us deal with the $\R^d$ case; however, the whole construction can be adapted also to the $d$-dimensional torus $\T^d$.

\medbreak
First of all, let us introduce the so called \textit{Littlewood-Paley decomposition}, based on a non-homogeneous dyadic partition of unity with
respect to the Fourier variable. 
We fix a smooth radial function $\chi$ supported in the ball $B(0,2)$, equal to $1$ in a neighborhood of $B(0,1)$
and such that $r\mapsto\chi(r\,e)$ is nonincreasing over $\R_+$ for all unitary vectors $e\in\R^d$. Set
$\varphi\left(\xi\right)=\chi\left(\xi\right)-\chi\left(2\xi\right)$ and
$\vphi_j(\xi):=\vphi(2^{-j}\xi)$ for all $j\geq0$.

The dyadic blocks $(\Delta_j)_{j\in\Z}$ are defined by\footnote{Throughout we agree  that  $f(D)$ stands for 
the pseudo-differential operator $u\mapsto\mc{F}^{-1}(f\,\mc{F}u)$.} 
$$
\Delta_j\,:=\,0\quad\mbox{ if }\; j\leq-2,\qquad\Delta_{-1}\,:=\,\chi(D)\qquad\mbox{ and }\qquad
\Delta_j\,:=\,\varphi(2^{-j}D)\quad \mbox{ if }\;  j\geq0\,.
$$
We  also introduce the following low frequency cut-off operator:
\begin{equation} \label{eq:S_j}
S_ju\,:=\,\chi(2^{-j}D)\,=\,\sum_{k\leq j-1}\Delta_{k}\qquad\mbox{ for }\qquad j\geq0\,.
\end{equation}
It is well known that for any $u\in\mc{S}'$,  one has the equality 
$$u\,=\,\sum_{j\geq-1}\Delta_ju\qquad\mbox{ in }\quad \mc{S}'\,.$$
Sometimes, we shall rather use the homogeneous cut-offs $\dot\Delta_j$ and $\dot S_j$, which are  defined by
$$
\dot\Delta_j\,:=\,\varphi(2^{-j}D)\qquad\mbox{ and }\qquad \dot S_j\,=\,\chi(2^{-j}D)\  \mbox{ for all }\  j\in\Z\,.
$$
Notice  that  we have $u\,=\,\sum_{j\in\Z}\dot\Delta_ju$
up to  polynomials only, which makes the previous decomposition unwieldy. A way to overcome that problem
is to restrict oneself to elements $u$  of  the set $\mc S_h'$ of tempered distributions such that
$$\lim_{j\ra-\infty}\bigl\|\dot S_ju\bigr\|_{L^\infty}\,=\,0\,.$$

Let us also mention the so-called \emph{Bernstein inequalities}, which explain the way derivatives act on spectrally localized functions.
  \begin{lemma} \label{l:bern}
Let  $0<r<R$.   A constant $C$ exists so that, for any nonnegative integer $k$, any couple $(p,q)$ 
in $[1,+\infty]^2$, with  $p\leq q$,  and any function $u\in L^p$,  we  have, for all $\lambda>0$,
$$
\displaylines{
{\rm supp}\, \widehat u \subset   B(0,\lambda R)\quad
\Longrightarrow\quad
\|\nabla^k u\|_{L^q}\, \leq\,
 C^{k+1}\,\lambda^{k+d\left(\frac{1}{p}-\frac{1}{q}\right)}\,\|u\|_{L^p}\;;\cr
{\rm supp}\, \widehat u \subset \{\xi\in\R^d\,|\, r\lambda\leq|\xi|\leq R\lambda\}
\quad\Longrightarrow\quad C^{-k-1}\,\lambda^k\|u\|_{L^p}\,
\leq\,
\|\nabla^k u\|_{L^p}\,
\leq\,
C^{k+1} \, \lambda^k\|u\|_{L^p}\,.
}$$
\end{lemma}   

By use of Littlewood-Paley decomposition, we can now define the class of Besov spaces.
\begin{defin} \label{d:B}
Let $s\in\R$ and $1\leq p,r\leq+\infty$.
\begin{enumerate}[(i)]
\item The \emph{non-homogeneous Besov space}
$B^{s}_{p,r}$ is the  set of tempered distributions $u$ for which
$$
\|u\|_{B^{s}_{p,r}}\,:=\,
\left\|\left(2^{js}\,\|\Delta_ju\|_{L^p}\right)_{j\geq-1}\right\|_{\ell^r}\,<\,+\infty\,.
$$
\item The \emph{homogeneous Besov space} $\dot B^s_{p,r}$ is the subset of distributions $u$ in $\mc S'_h$ such that
$$\|u\|_{\dot B^{s}_{p,r}}\,:=\,\left\|\left(2^{js}\,\|\dot\Delta_ju\|_{L^p}\right)_{j\in\Z}\right\|_{\ell^r}\,<\,+\infty\,.$$
\end{enumerate}
\end{defin}

It is well known that, for all $s\in\R$, $B^s_{2,2}$ coincides with $H^s$, with equivalent norms:
\begin{equation} \label{eq:LP-Sob}
\|f\|_{H^s}\,\sim\,\left(\sum_{j\geq-1}2^{2 j s}\,\|\Delta_jf\|^2_{L^2}\right)^{1/2}\,.
\end{equation}
An analogous property holds also for the homogeneous spaces $\dot H^s$.
When $p\neq 2$, nonhomogeneous (resp. homogeneous) Besov spaces are interpolation spaces between Sobolev spaces $W^{k,p}$ (resp. $\dot W^{k,p}$):
for all $p\in\,]1,+\infty[\,$, one has the following continuous embeddings:
$$\dot B^0_{p,\min(p,2)}\,\hookrightarrow\, L^p\,\hookrightarrow\, \dot B^0_{p,\max(p,2)}
\qquad\mbox{ and }\qquad B^0_{p,\min(p,2)}\,\hookrightarrow\, L^p\,\hookrightarrow\, B^0_{p,\max(p,2)}\,.$$

As an immediate consequence of the Bernstein inequalities, one gets the following Sobolev-type embedding result.
\begin{prop}\label{p:embed}
Let $1\leq p_1\leq p_2\leq+\infty.$ 
The, the space $B^{s_1}_{p_1,r_1}$ is continuously embedded in the space $B^{s_2}_{p_2,r_2}$ whenever
$$
s_2\,<\,s_1-d\left(\frac{1}{p_1}-\frac{1}{p_2}\right)\qquad\mbox{ or }\qquad
s_2\,=\,s_1-d\left(\frac{1}{p_1}-\frac{1}{p_2}\right)\;\;\mbox{ and }\;\;r_1\,\leq\,r_2\,. 
$$
The space $\dot B^{s_1}_{p_1,r_1}$ is continuously embedded in the space $\dot B^{s_2}_{p_2,r_2}$ if
$$
s_2\,=\,s_1-d\left(\frac{1}{p_1}-\frac{1}{p_2}\right)\;\;\mbox{ and }\;\;r_1\,\leq\,r_2\,. 
$$
\end{prop}

We recall also Lemma 2.73 of \cite{B-C-D}.
\begin{lemma} \label{l:Id-S}
If $1\leq r<+\infty$, for any $f\in B^s_{p,r}$ one has
$$
\lim_{j\ra+\infty}\left\|f\,-\,S_jf\right\|_{B^s_{p,r}}\,=\,0\,.
$$
\end{lemma}

Before going on, let us introduce also the so-called \emph{Chemin-Lerner spaces} (defined first in \cite{Ch-Ler}). They are time-dependent Besov spaces, where the time integration
is performed before the $\ell^r$ summation. See also Paragraph 2.6.3 of \cite{B-C-D} for more details.
\begin{defin} \label{d:C-L}
Let $s\in\R$, the triplet $(q,p,r)\in[1,+\infty]^3$ and $T\in[0,+\infty]$. The space $\wtilde{L}^q_T(B^s_{p,r})$ is defined as the set of tempered distributions
$u\in \mc S'\big([0,T[\,\times\R^d\big)$ such that
$$
\|u\|_{\wtilde L^q_T(B^s_{p,r})}\,:=\,\Bigl\| \Bigl(  2^{js}  \|\Delta_j u(t)\|_{L^q_T(L^p)} \Bigr)_{j\geq -1}\Bigr\|_{\ell^r}\,<\,+\infty\,.
$$
We also set $\wtilde C_T(B^s_{p,r})=\wtilde L_T^\infty(B^s_{p,r})\cap C\big([0,T];B^s_{p,r}\big)$.
\end{defin}
The relation between these classes and the classical $L^q_T(B^s_{p,r})$ can be easily recovered by Minkowski's inequality:
$$
\left\{\begin{array}{lcl}
        \|u\|_{\wtilde{L}^q_T(B^s_{p,r})}\;\leq\;\|u\|_{L^q_T(B^s_{p,r})} & \mbox{ if } & q\,\leq\,r \\[1ex]
\|u\|_{\wtilde{L}^q_T(B^s_{p,r})}\;\geq\;\|u\|_{L^q_T(B^s_{p,r})} & \mbox{ if } & q\,\geq\,r\,.
       \end{array}\right.
$$
We will need those spaces in Paragraph \ref{sss:sigma-reg}.

\medbreak
Let us now introduce the paraproduct operator (after J.-M. Bony, see \cite{Bony}). Constructing the paraproduct operator relies on the observation that, 
formally, any product  of two tempered distributions $u$ and $v$ may be decomposed into 
\begin{equation}\label{eq:bony}
u\,v\;=\;T_uv\,+\,T_vu\,+\,R(u,v)\,,
\end{equation}
where we have defined
$$
T_uv\,:=\,\sum_jS_{j-1}u\Delta_j v,\qquad\qquad\mbox{ and }\qquad\qquad
R(u,v)\,:=\,\sum_j\sum_{|j'-j|\leq1}\Delta_j u\,\Delta_{j'}v\,.
$$
The above operator $T$ is called ``paraproduct'' whereas
$R$ is called ``remainder''.
The paraproduct and remainder operators have many nice continuity properties. 
The following ones are of constant use in this paper (see the proof in e.g. Chapter 2 of \cite{B-C-D}).
\begin{prop}\label{p:op}
For any $(s,p,r)\in\R\times[1,+\infty]^2$, $(r_1,r_2)\in[1,+\infty]^2$ and $t>0$, the paraproduct operator 
$T$ maps continuously $L^\infty\times B^s_{p,r}$ in $B^s_{p,r}$ and  $B^{-t}_{\infty,r_1}\times B^s_{p,r_2}$ in $B^{s-t}_{p,r_3}$, where $1/r_3\,:=\,\min\big\{1,1/r_1+1/r_2\big\}$.
Moreover, the following estimates hold:
$$
\|T_uv\|_{B^s_{p,r}}\,\leq\, C\,\|u\|_{L^\infty}\,\|\nabla v\|_{B^{s-1}_{p,r}}\qquad\mbox{ and }\qquad
\|T_uv\|_{B^{s-t}_{p,r_3}}\,\leq\, C\|u\|_{B^{-t}_{\infty,r_1}}\,\|\nabla v\|_{B^{s-1}_{p,r_2}}\,.
$$
For any $(s_1,p_1,r_1)$ and $(s_2,p_2,r_2)$ in $\R\times[1,+\infty]^2$ such that 
$s_1+s_2>0$, $1/p:=1/p_1+1/p_2\leq1$ and~$1/r:=1/r_1+1/r_2\leq1$,
the remainder operator $R$ maps continuously~$B^{s_1}_{p_1,r_1}\times B^{s_2}_{p_2,r_2}$ into~$B^{s_1+s_2}_{p,r}$.
In the case $s_1+s_2=0$, provided $r=1$, operator $R$ is continuous from $B^{s_1}_{p_1,r_1}\times B^{s_2}_{p_2,r_2}$ with values
in $B^{0}_{p,\infty}$.
\end{prop}
It goes without saying that similar properties hold true also in the class of homogeneous Besov spaces.
As a corollary of the previous proposition, we deduce some continuity properties of the product in Sobolev spaces, which will be used in the course of our analysis.
In the case $d=3$, we get the next statement.
\begin{coroll} \label{c:product}
Let $d=3$.
\begin{enumerate}[(i)]
 \item For all $0<\sigma<3/2$ and all $0<\g<\sigma$, the product maps continuously $H^{-\g}\times H^\sigma$ into the space $H^{\sigma-\g-3/2}$.
 \item The product is a continuous map from $H^1\times H^1$ into $H^{1/2}$.
\end{enumerate}
\end{coroll}

When $d=2$, instead, we get the following result.
\begin{coroll} \label{c:product-2}
Let $d=2$.
\begin{enumerate}[(i)]
 \item For all $\eta$ and all $\delta$ in $\,]0,1[\,$, such that $1-\eta-\delta>0$, the product is a continuous map from $H^{-\eta}\times H^{1-\delta}$ into $H^{-\eta-\delta}$.
 \item For all $-1<\eta<1$, the product is a continuous map from $H^{\eta}\times H^{1}$ into $H^{\eta-\delta}$ for all $\delta>0$ arbitrarily small.
\item For all $-2<\eta<2$, the product is a continuous map from $H^{\eta}\times H^{2}$ into $H^{\eta}$.
\item The product is a continuous map from $H^{1}\times H^{1}$ into $H^{1-\delta}$ for all $\delta>0$ arbitrarily small.
\end{enumerate}
\end{coroll}

The proof to Corollary \ref{c:product-2} can be found in e.g. \cite{F-G_RMI}. Therefore, let us only show the proof to Corollary \ref{c:product}.
\begin{proof}[Proof of Corollary \ref{c:product}]
Let us take two tempered distributions $a\in H^{-\g}$ and~$w\in H^{\sigma}$, and write
\begin{equation} \label{eq:paraprod_a-w}
a\,w\,=\,T_aw\,+\,T_wa\,+\,R(a,w)\,.
\end{equation}
By a systematic use of Proposition \ref{p:op} and of embeddings $H^s\hra B^{s-3/2}_{\infty,\infty}$, we deduce that
$$
\left\|T_aw+T_wa\right\|_{H^{\sigma-\g-3/2}}\,+\,\left\|R(a,w)\right\|_{B^{\sigma-\g}_{1,1}}\,\leq\,C\,\|a\|_{H^{-\g}}\,\|w\|_{H^{\sigma}}\,.
$$
At this point, the continuous embedding $B^{\sigma-\g}_{1,1}\,\hra\,H^{\sigma-\g-3/2}$, which follows from Proposition \ref{p:embed}, completes the proof of our claim.

Let us switch to the proof of claim (ii). Let us use \eqref{eq:paraprod_a-w} again, where now $a$ and $w$ belong both to $H^1$.
Notice that Proposition \ref{p:embed} implies the embedding $H^1\hookrightarrow B^{-1/2}_{\infty,\infty}$: then, from Proposition~\ref{p:op} we infer that $T_aw$ and $T_wa$ both belong
to $H^{1/2}$. On the other hand, the same Proposition~\ref{p:op} implies $R(a,w)\in B^{2}_{1,1}\,\hookrightarrow\,H^{1/2}$, where we have used Proposition \ref{p:embed} again.

The corollary is now proved.
\end{proof}

To conclude, let us recall Gagliardo-Nirenberg inequalities, which we will repeatedly use in our analysis. We refer e.g. to Corollary 1.2 of \cite{C-D-G-G}
for their proof.
\begin{prop} \label{p:Gagl-Nir}
Let $p\in[2,+\infty[\,$ such that $1/p\,>\,1/2\,-\,1/d$. There exists a constant $C>0$ such that, for any domain $\Omega\subset\R^d$
and for all $u\in H^1_0(\Omega)$, the following inequality holds true:
$$
\|u\|_{L^p(\Omega)}\,\leq\,C\,\|u\|_{L^2(\Omega)}^{1-\lam}\;\|\nabla u\|_{L^2(\Omega)}^{\lam}\,,\qquad\qquad\mbox{ with }\qquad
\lam\,=\,\frac{d\,(p-2)}{2\,p}\,.
$$
\end{prop}

\subsection{Heat equation with fast diffusion} \label{app:heat}
In this subsection, we prove decay estimates for the derivatives of parabolic-type equations with fast diffusion in time.
More precisely, let $\nu:[0,1]\,\longrightarrow\,[0,1]$ be a continuous, strictly increasing function such that $\nu(0)=0$.
For all $0<\veps<1$, consider the family of heat equations
\begin{equation} \label{eq:fast-heat}
\left\{\begin{array}{l}
        \d_t\Phi_\veps\,-\,\dfrac{1}{\nu(\veps)}\,\Delta\Phi_\veps\,=\,g_\veps \\[1ex]
        \big(\Phi_\veps\big)_{|t=0}\,=\,\Phi_{0,\veps}\,,
       \end{array}
\right.
\end{equation}
where the sequences $\big(\Phi_{0,\veps}\big)_\veps$ and $\big(g_\veps\big)_\veps$ are uniformly bounded respectively in the space $H^\infty(\R^2)\,:=\,\bigcap_{s\in\R}H^s(\R^2)$ and
$L^2\big([0,T];H^\infty(\R^2)\big)$, for all $T>0$. Notice that we can represent the smooth solution $\Phi_\veps$ according to Duhamel's formula as
\begin{equation} \label{eq:duhamel}
\Phi_{\veps}(t,x)\,=\,e^{\Delta\,t/\nu(\veps)}\,\Phi_{0,\veps}(x)\,+\,\int^t_0e^{\Delta\,(t-\tau)/\nu(\veps)}\,g(\tau,x)\,d\tau\,.
\end{equation}

It is well-known that the solutions to a linear heat equation decay in time, in suitable norms. Then, we expect that the solutions $\Phi_\veps$ to \eqref{eq:fast-heat}, together with their derivatives,
decay to $0$ when $\veps\ra0^+$: we need a precise quantitative estimate for the norms of the higher order derivatives.

Notice that finding the exact rate in terms of $\veps$ is the key for the analysis of Subsection \ref{sss:strong-vort}; therefore, we will have to face the difficulty of handling the lack
of time integrability near $0$. This is also the main reason why we prefer not to rescale the time variable. Hence, previous results (see e.g. \cite{Z}; see also \cite{DeA-F} and the references therein)
on the long-time behaviour of solutions to parabolic equations are not useful in our context.

\begin{thm} \label{t:fast-heat}
Let $\big(\Phi_{0,\veps}\big)_\veps\,\subset\, H^\infty(\R^2)$ and $\big(g_\veps\big)_\veps\,\subset\,L^2_{\rm loc}\big(\R_+;H^\infty(\R^2)\big)$.
For all $\veps\in\,]0,1]$, let $\Phi_\veps$ be the smooth solution to the Cauchy problem \eqref{eq:fast-heat}.
Let $s\geq1$ and $T>0$ be fixed.

Then, for any $\de\in\,]0,1[\,$ fixed, there exists a constant $C=C(T,s,\de)$ and a number $\veps_0=\veps_0(s,\de)$, such that, for all $\veps\leq\veps_0$, one has the estimate
$$
\left\|(-\Delta)^s\Phi_\veps\right\|^2_{L^2\big(\,]\de,T[\,;L^2\big)}\,\leq\,C\,\big(\nu(\veps)\big)^s\left(\left\|\Phi_{0,\veps}\right\|^2_{L^2}\,+\,\left\|g_\veps\right\|^2_{L^2_T(H^s)}\right)\,.
$$
\end{thm}

\begin{proof}
Applying the operator $(-\Delta)^s$ to equation \eqref{eq:duhamel}, we find
$$
(-\Delta)^s\Phi_{\veps}(t,x)\,=\,e^{\Delta\,t/\nu(\veps)}\,(-\Delta)^s\Phi_{0,\veps}(x)\,+\,\int^t_0e^{\Delta\,(t-\tau)/\nu(\veps)}\,(-\Delta)^sg(\tau,x)\,d\tau\,.
$$
Therefore, we need to estimate the initial datum term and the forcing term separately, the latter being harder, since one has to deal with the time integral.

Let us start with the term containing the initial datum: by Plancherel theorem we have
\begin{align*}
\left\|e^{\Delta\,t/\nu(\veps)}\,(-\Delta)^s\Phi_{0,\veps}\right\|^2_{L^2}\,&=\,C\,\int_{\R^2}e^{-2|\xi|^2\,t/\nu(\veps)}\,|\xi|^{2s}\,\left|\what\Phi_{0,\veps}(\xi)\right|^2\,d\xi \\
&=\,\frac{C_s\,\nu^s(\veps)}{t^s}\,\int_{\R^2}e^{-2|\xi|^2\,t/\nu(\veps)}\left(\frac{2\,t\,|\xi|^2}{\nu(\veps)}\right)^{\!\!s}\,\left|\what\Phi_{0,\veps}(\xi)\right|^2\,d\xi \\
&\leq\,\frac{C_s\,\nu^s(\veps)}{t^s}\,\left\|\Phi_{0,\veps}\right\|^2_{L^2}\,.
\end{align*}
Therefore, given $0<\de<1$, integrating in time over $\,]\de,T[\,$ yields
\begin{equation} \label{est:heat_0}
\left\|e^{\Delta\,t/\nu(\veps)}\,(-\Delta)^s\Phi_{0,\veps}\right\|^2_{L^2\big(\,]\de,T[\,;L^2\big)}\,\leq\,C_s\,\frac{\nu^s(\veps)}{\de^{s-1}}\,\left\|\Phi_{0,\veps}\right\|^2_{L^2}\,,
\end{equation}
where the factor $1/\de^{s-1}$ has to be replaced by $-\log\de$ when $s=1$.

Let us now deal with the forcing term: using Plancherel theorem again, we find
\begin{align*}
&\left\|\int^t_0e^{\Delta\,(t-\tau)/\nu(\veps)}\,(-\Delta)^sg(\tau)\,d\tau\right\|^2_{L^2\big(\,]\de,T[\,;L^2\big)} \\
&\qquad\qquad=\,C\,\int^T_\de\int_0^t\int_{\R^2}
e^{-2|\xi|^2\,(t-\tau)/\nu(\veps)}\,|\xi|^{2s}\,\left|\what g_\veps(\tau,\xi)\right|^2\,d\xi\,d\tau\,dt\,=\,C\,\big(J_1\,+\,J_2\,+\,J_3\big)\,,
\end{align*}
where we have defined, for $a\leq\de/2$ to be chosen later,
\begin{align*}
J_1\,&:=\,\int^T_\de\int_{0}^{a}\int_{\R^2}e^{-2|\xi|^2\,(t-\tau)/\nu(\veps)}\,|\xi|^{2s}\,\left|\what g_\veps(\tau,\xi)\right|^2\,d\xi\,d\tau\,dt \\
J_2\,&:=\,\int^T_\de\int_{a}^{t-a}\int_{\R^2}e^{-2|\xi|^2\,(t-\tau)/\nu(\veps)}\,|\xi|^{2s}\,\left|\what g_\veps(\tau,\xi)\right|^2\,d\xi\,d\tau\,dt \\
J_3\,&:=\,\int^T_\de\int_{t-a}^t\int_{\R^2}e^{-2|\xi|^2\,(t-\tau)/\nu(\veps)}\,|\xi|^{2s}\,\left|\what g_\veps(\tau,\xi)\right|^2\,d\xi\,d\tau\,dt\,.
\end{align*}

Let us start by considering the integral $J_1$: we remark that, by our choice of $a$, we have $t-\tau\,\geq\,\de/2$. Therefore, arguing similarly as for the term containing the initial datum, we get
\begin{align} \label{est:J_1}
\left|J_1\right|\,&\leq\,C\int^T_\de\int^a_0\int_{\R^2}e^{-\de\,|\xi|^2/\nu(\veps)}\,|\xi|^{2s}\,\left|\what g_\veps(\tau,\xi)\right|^2\,d\xi\,d\tau\,dt\;\leq\;
C_s\,\frac{\nu^s(\veps)}{\de^s}\,T\,\left\|g_\veps\right\|_{L^2_T(L^2)}\,.
\end{align}

As for $J_2$, the argument is pretty similar: to begin with, one computes
\begin{align*}
|J_2|\,&\leq\,C\,\nu^s(\veps)\int^T_\de\int^{t-a}_a\int_{\R^2}\frac{1}{(t-\t)^s}\,e^{-2|\xi|^2\,(t-\tau)/\nu(\veps)}\,\left(\frac{t-\t}{\nu(\veps)}|\xi|^{2}\right)^{\!\!s}\,
\left|\what g_\veps(\tau,\xi)\right|^2\,d\xi\,d\tau\,dt \\
&\leq\,C_s\,\nu^s(\veps)\int^T_\de\int_a^{t-a}\frac{1}{(t-\tau)^s}\,\left\|g_\veps(\tau)\right\|^2_{L^2}\,d\tau\,dt\,.
\end{align*}
Hence, Young inequality for convolutions implies, for $s>1$
\begin{equation} \label{est:J_2}
\left|J_2\right|\,\leq\,C_s\,\frac{\nu^s(\veps)}{\de^{s-1}}\,\left\|g_\veps\right\|^2_{L^2_T(L^2)}\,,
\end{equation}
where the factor $1/\de^{s-1}$ has to be replaced by $-\log\de$ when $s=1$, as before.

Finally, let us deal with $J_3$: first of all, observe that, after a change of variable in the second integral, we can write
\begin{align*}
J_3\,&=\,\int^T_\de\int_{0}^a\int_{\R^2}e^{-2|\xi|^2\,\tau/\nu(\veps)}\,|\xi|^{2s}\,\left|\what g_\veps(t-\tau,\xi)\right|^2\,d\xi\,d\tau\,dt\,.
\end{align*}
At this point, bounding the exponential term by $1$ and inverting the order of the time integrals, we gather
\begin{align*}
\left|J_3\right|\,&\leq\,\int^T_\de\int_0^a\left\|g_\veps(t-\tau)\right\|^2_{H^s}\,d\tau\,dt\;\leq\;\int^a_0\int^T_\de\left\|g_\veps(t-\tau)\right\|^2_{H^s}\,dt\,d\tau\,,
\end{align*}
which finally yields the estimate
\begin{equation} \label{est:J_3}
\left|J_3\right|\,\leq\,C\,a\,\left\|g_\veps\right\|^2_{L^2_T(H^s)}\,.
\end{equation}
To conclude, we make the choice $a=\nu^s(\veps)$, which requires to introduce the constraint $\veps\leq\veps_\de$, where $\veps_\de$ is such that
$$
\nu^s(\veps_\de)\,\leq\,\de/2\,.
$$

Putting together estimates \eqref{est:heat_0}, \eqref{est:J_1}, \eqref{est:J_2} and \eqref{est:J_3} completes the proof of the theorem.
\end{proof}

\section{Study of the singular perturbation} \label{s:singular}
In this section we study preliminary properties for tackling the singular perturbation problem. In a first time, we derive uniform bounds for the family of weak solutions $\big(\rho_\veps,u_\veps\big)_\veps$.
Those bounds allow us to identify weak-limit points $(\rho,u)$: in Subsection \ref{ss:constraints} we then derive constraints $(\rho,u)$ has to satisfy.
Finally, in Subsection \ref{ss:further} we come back to the mass equation, and infer further properties and bounds for the density functions.

\subsection{Uniform bounds and first convergence properties} \label{ss:unif-bounds}

In this subsection, we derive uniform bounds for the family $\big(\rho_\veps,u_\veps\big)_\veps$. All the bounds come from the energy inequality \eqref{est:energy},
which is satisfied by assumption.

For this, following a classical approach (see e.g. \cite{F-N}), it is convenient to introduce a decomposition of any function $h$ in its \emph{essential} and \emph{residual} parts.
To begin with, 
for almost every time $t>0$ and all $\veps\in\,]0,1]$, we define the sets
$$
\Omega_\ess^\veps(t)\,:=\,\left\{x\in\Omega\;\big|\quad \rho_\veps(t,x)\in\left[1/2\,,\,2\right]\right\}\,,\qquad\Omega^\veps_\res(t)\,:=\,\Omega\setminus\Omega^\veps_\ess(t)\,.
$$
Then, given a function $h$, we write
$$
h\,=\,\left[h\right]_\ess\,+\,\left[h\right]_\res\,,\qquad\qquad\mbox{ where }\qquad \left[h\right]_\ess\,:=\,h\,\mathds{1}_{\Omega_\ess^\veps(t)}\,.
$$
Here above, $\mathds{1}_A$ denotes the characteristic function of a set $A\subset\Omega$.

We are now ready to establish uniform bounds.
First of all, we remark that, in view of the structure of the initial data, the right-hand side of \eqref{est:energy} is bounded, uniformly in $\veps>0$. Then, we immediately deduce
that
\begin{equation} \label{ub:u-Du}
\left(\sqrt{\rho_\veps}\,u_\veps\right)_\veps\,\subset\,L^\infty\big(\R_+;L^2(\Omega)\big)\qquad\mbox{ and }\qquad \left(\nabla u_\veps\right)_\veps\,\subset\,L^2\big(\R_+;L^2(\Omega)\big)\,,
\end{equation}
together with the bound
\begin{equation} \label{ub:div-u}
\left(\frac{1}{\veps^\beta}\;\div u_\veps\right)_\veps\,\subset\,L^2\big(\R_+;L^2(\Omega)\big)\,.
\end{equation}
For later use, let us introduce $\theta\in L^2\big(\R_+;L^2(\Omega)\big)$ to be the function such that
\begin{equation} \label{cv:div-u}
\frac{1}{\veps^\beta}\;\div u_\veps\,\rightharpoonup\,\theta\qquad\qquad \mbox{ in }\qquad L^2\big(\R_+;L^2(\Omega)\big)\,.
\end{equation}

Next, from the relative entropy functional it is customary to get (see e.g. \cite{F-N} for details)
\begin{align}
 \sup_{t\in\R_+}\left\|\frac{1}{\veps^\alpha}\;\left[\rho_\veps-1\right]_\ess\right\|_{L^2(\Omega)}\,&\leq\,C  \label{ub:dens-ess} \\
 \sup_{t\in\R_+}\left\|\left[\rho_\veps\right]_\res\right\|^\g_{L^\g(\Omega)}\,+\,\sup_{t\in\R_+}\left\|\left[1\right]_\res\right\|_{L^1(\Omega)}\,&\leq\,C\,\veps^{2\alpha}\,. \label{ub:dens-res}
\end{align}
Observe that all those bounds hold true also in the endpoint case $\alpha=0$.

\subsubsection{Additional bounds when $0<\alpha<1$} \label{sss:ub-a}
Let us restrict for a while to the case $0<\alpha<1$. From \eqref{ub:dens-ess} and \eqref{ub:dens-res}, we can write
\begin{align} \label{eq:rho-decomp}
\rho_\veps\,-\,1\,=\,\rho_\veps^{(1)}\,+\,\rho_\veps^{(2)}\,, 
\end{align}
where, for all $T>0$, one has $\rho_\veps^{(1)}\longrightarrow0$ in $L^\infty_T(L^2)$ and $\rho_\veps^{(2)}\longrightarrow0$ in $L^\infty_T(L^\g)$.
From the previous decomposition, arguing as in \cite{F-G-N} (see also Paragraph \ref{sss:ub-0} below) it is easy to get
\begin{equation} \label{ub:u}
\big(u_\veps\big)_\veps\,\subset\,L^2_{\rm loc}\big(\R_+;L^2(\Omega)\big)\,.
\end{equation}
Therefore, there exists a $u\in L^2_{\rm loc}\big(\R_+;H^1(\Omega)\big)$ such that, up to an extraction,
\begin{equation} \label{cv:u}
u_\veps\,\rightharpoonup\,u\qquad\qquad\mbox{ in }\qquad L^2_{\rm loc}\big(\R_+;H^1(\Omega)\big)\,.
\end{equation}
Using \eqref{eq:rho-decomp} and Sobolev embeddings, we also get that
$\big(\rho_\veps\,u_\veps\big)_\veps$ is uniformly bounded in $L^2_T(L^2+L^{3/2}+L^{6\g/(\g+6)})$ for all $T>0$, and 
\begin{equation}\label{cv:rho-u}
\rho_\veps\,u_\veps\,\rightharpoonup\,u\qquad\qquad\mbox{ in }\qquad L^2_T(L^2+L^{3/2}+L^{6\g/(\g+6)})\,.
\end{equation}

Next, let us define the quantity
$$
r_\veps\,:=\,\frac{1}{\veps^\alpha}\,\big(\rho_\veps\,-\,1\big)\,.
$$
From the uniform bound \eqref{ub:dens-ess}, we immediately deduce (omitting the extraction of a suitable subsequence) that
\begin{equation} \label{cv:r_ess}
\big[r_\veps\big]_\ess\,\stackrel{*}{\rightharpoonup}\,r\qquad\qquad\mbox{ in }\qquad L^\infty\big(\R_+;L^2(\Omega)\big)\,,
\end{equation}
for some $r$ belonging to that space.
On the other hand, in view of \eqref{ub:dens-res}, we can bound
\begin{equation} \label{ub:r_res}
\int_\Omega\left|\big[r_\veps\big]_\res\right|^\g\,dx\,\leq\,\frac{1}{\veps^{\g\,\alpha}}\left(\int_\Omega\left|\big[\rho_\veps\big]_\res\right|^\g\,dx\,+\,\int_\Omega[1]_\res\,dx\right)\,\leq\,
C\,\veps^{\alpha(2-\g)}\,,
\end{equation}
which immediately implies that
\begin{equation} \label{cv:r_res}
\big[r_\veps\big]_\res\,\longrightarrow\,0\qquad\qquad\mbox{ in }\qquad L^\infty\big(\R_+;L^p(\Omega)\big)\,,\qquad \forall\;1\leq p<\min\{2,\g\}\,.
\end{equation}

\subsubsection{Additional bounds when $\alpha=0$} \label{sss:ub-0}
Now, we consider the case $\alpha=0$. Recall that, in this case, we restrict our attention to the $2$-dimensional domain $\R^2$, and we assume $\g>1$ in \eqref{hyp:P}.

When $\alpha=0$, we still dispose of estimates \eqref{ub:dens-ess}, \eqref{ub:dens-res}, but they do not give any smallness property
for the density variations. Nonetheless, using \eqref{eq:rho-decomp} again and arguing as in \cite{F-G-N}, we are able to establish also in the case $\alpha=0$ the uniform boundedness property \eqref{ub:u}.
Indeed, first of all we write
\begin{equation} \label{est:u-L^2_start}
\int_{\R^2}|u_\veps|^2\,dx\,\leq\,\int_{\R^2}\rho_\veps\,|u_\veps|^2\,dx\,+\,\int_{\R^2}\bigl|\rho_\veps-1\bigr|\,|u_\veps|^2\,dx\,,
\end{equation}
where the former term in the right-hand side is uniformly bounded in $L^\infty(\R_+;L^2)$ in view of \eqref{ub:u-Du}.
For the latter term, we can use the decomposition \eqref{eq:rho-decomp}, where $\big(\rho_\veps^{(1)}\big)_\veps$ and $\big(\rho_\veps^{(2)}\big)_\veps$
are uniformly bounded in $L^\infty(\R_+;L^2)$ and $L^\infty(\R_+;L^\g)$ respectively. On the one hand, by H\"older and Gagliardo-Nirenberg inequalities, we can estimate
\begin{align}
\int_{\R^2}\left|\rho_\veps^{(1)}\right|\,|u_\veps|^2\,&\leq\,\left\|\rho_\veps^{(1)}\right\|_{L^2}\,
\left\|u_\veps\right\|^2_{L^4}\,\leq\,C\,\left\|u_\veps\right\|_{L^2}\,\left\|\nabla u_\veps\right\|_{L^2}\,. \label{est:u-L^2_1}
\end{align}
On the other hand, after defining $\g'$ such that $1/\g\,+\,1/\g'\,=\,1$, thanks to \eqref{ub:dens-res} we infer
\begin{align*}
\int_{\R^2}\left|\rho_\veps^{(2)}\right|\,|u_\veps|^2\,&\leq\,\left\|\rho_\veps^{(2)}\right\|_{L^\g}\,\left\|u_\veps\right\|^2_{L^{2\g'}}\,. 
\end{align*}
Notice that, since $\g>1$, we have $1<\g'<+\infty$, hence we can apply Gagliardo-Nirenberg inequality again: we find
\begin{align} \label{est:u-L^2_2}
\int_{\R^2}\left|\rho_\veps^{(2)}\right|\,|u_\veps|^2\,&\leq\,C\,\left(\left\|u_\veps\right\|^2_{L^{2}}\right)^{\!1-1/\g}\,
\left(\left\|\nabla u_\veps\right\|_{L^2}^{2}\right)^{\!1/\g}\,.
\end{align}
Therefore, inserting \eqref{est:u-L^2_1} and \eqref{est:u-L^2_2} into \eqref{est:u-L^2_start} and applying Young inequality,
we finally deduce the claimed estimate: there exists a constant $C>0$ such that, for all $\veps>0$ and all $T>0$, one has
\begin{equation} \label{est:u-L^2}
\left\|u^\veps\right\|_{L^2_T(L^2)}\,\leq\,C\,.
\end{equation}
The previous bound immediately implies \eqref{cv:u}, as in the previous paragraph.

\medbreak
Next, let us turn our attention to the density fluctuations: properties like \eqref{cv:r_ess} and \eqref{ub:r_res} are not very useful in the case $\alpha=0$.
Therefore, let us argue in a different way.

Resorting to the decomposition \eqref{eq:rho-decomp} again, we see that $\big(\rho_\veps-1\big)_\veps\,\subset\,L^\infty_T(L^2)$ if $\g\geq2$. On the other hand, when $1<\g<2$,
dual Sobolev embeddings imply that
\begin{equation} \label{ub:rho_e-trho}
\left(\rho_\veps\,-\,1\right)_\veps\,\subset\,L^\infty_T(H^{-s_1})\,,\qquad\qquad\mbox{ where }\qquad s_1\,:=\,(2-\g)/\g\,,
\end{equation}
for all fixed time $T>0$. From now on, we will use \eqref{ub:rho_e-trho} for any $\g>1$, with the convention that
$s_1=0$ if $\g\geq2$. Therefore, there exists a function $r\in L^\infty_{\rm loc}\big(\R_+;H^{-s_1}(\Omega)\big)$  such that, up to an extraction,
\begin{equation} \label{cv:rho-e-trho}
\rho_\veps\,-\,1\,:=\,r_\veps\,\stackrel{*}{\rightharpoonup}\,r\qquad\qquad\mbox{ in }\qquad\qquad L^\infty_{\rm loc}\big(\R_+;H^{-s_1}(\Omega)\big)\,.
\end{equation}
Furthermore, we can write
$$ 
\rho_\veps\,u_\veps\,=\,\left(\left[\sqrt{\rho_\veps}\right]_\ess\,+\,\left[\sqrt{\rho_\veps}\right]_\res\right)\,\sqrt{\rho_\veps}\,u_\veps\,,
$$ 
where $\big(\left[\sqrt{\rho_\veps}\right]_\ess\big)_\veps$ is uniformly bounded in time and space (by definition of essential set), while
$\big(\left[\sqrt{\rho_\veps}\right]_\res\big)_\veps$ is uniformly bounded in the space $L^\infty_T(L^{2\g})$ for any $T>0$. As a consequence, after defining
$p$ such that $1/p\,=\,1/2\,+\,1/(2\g)$, by dual Sobolev embeddings we get that 
\begin{equation} \label{ub:V-2D}
\big(V_\veps\big)_\veps\,\subset\,L^\infty_T(L^2\,+\,L^p)\,\hookrightarrow\,L^\infty_T(H^{-s_2})\,,\qquad\qquad\mbox{ where }\qquad s_2\,:=\,1/\g\,.
\end{equation}
for any $T>0$ fixed.
Therefore, an easy inspection of the mass equation in \eqref{eq:sing-NSC_2D} reveals that
$\big(\d_tr_\veps\big)_\veps\,\subset\,L^\infty_T(H^{-1-s_2})$,
which immediately implies that, for any $T>0$ fixed, one has
\begin{equation} \label{ub:rho_e-trho_T}
\left(r_\veps\right)_\veps\,\subset\,W^{1,\infty}_T(H^{-1-s_2})\,.
\end{equation}
Putting \eqref{ub:rho_e-trho} and \eqref{ub:rho_e-trho_T} together and applying Ascoli-Arzel\`a theorem, we get, up to a further extraction that we omit,
the strong convergence $r_\veps\,\longrightarrow\,r$ when $\veps\ra0^+$ in the space $\mc C\big([0,T];H^{-1-s_2-\de}_{\rm loc}\big)$, for any $\de>0$.
Interpolation with the previous uniform bounds finally yields the strong convergence
\begin{equation} \label{cv:s_e}
r_\veps\,\longrightarrow\,r\qquad\qquad \mbox{ in }\qquad \mc C^{0,1-\eta}\big([0,T];H_{\rm loc}^{-s_2-1-\de+\eta(-s_1+s_2+1+\de)}\big)\,,
\end{equation}
for all $0<\eta<1$ and all $T>0$, where $\de>0$ is arbitrarily small.

From now on, whenever $\alpha=0$, we use the notation
$$
\rho_\veps(t,x)\,=\,1\,+\,r_\veps(t,x)\qquad\qquad\mbox{ and }\qquad\qquad
\rho(t,x)\,=\,1\,+\,r(t,x)\,.
$$

\subsection{Constraints on the limit} \label{ss:constraints}

In the previous part we have proved uniform bounds on the sequence of weak solutions $\bigl(\rho_\veps,u_\veps\bigr)_\veps$, which allow us to identify (up to extraction)
weak limits $(\rho,u)$. In the present subsection, we collect some properties these limit-points have to satisfy. 
We point out that these conditions do not fully characterise  the limit dynamics. 

\subsubsection{The case $0<\alpha<1$} \label{sss:constr-a}

To begin with, let us consider the case $0<\alpha<1$. We start with a simple lemma, which shows that the pressure term is of order $O(\veps^{-\alpha})$.
\begin{lemma} \label{l:p}
Let $0\leq\alpha<1$. Then we can write, in the sense of $\mc D'$,
$$
\frac{1}{\veps^{2\alpha}}\,\nabla P(\rho)\,=\,\frac{1}{\veps^\alpha}\,P'(1)\,\nabla r_\veps\,+\,\frac{1}{\veps^{2\alpha}}\,\nabla\Pi\big(\rho_\veps,1\big)\,,
$$
where we have defined $\Pi\big(\rho,1\big)\,:=\,P(\rho)\,-\,P(1)\,-\,P'(1)\,\big(\rho-1\big)$. Moreover, one has
$$
\left(\frac{1}{\veps^{2\alpha}}\,\Pi\big(\rho_\veps,1\big)\right)_\veps\,\subset\,L^\infty\big(\R_+;L^2+L^1(\Omega)\big)\,.
$$
\end{lemma}

\begin{proof}
It is enough to prove the uniform bound for the function $\Pi\big(\rho_\veps,1\big)$. For this, we resort to the decomposition into essential and residual parts.

First of all, by a Taylor expansion, we have
$$
\left|\frac{1}{\veps^{2\alpha}}\,\left[\Pi\big(\rho_\veps,1\big)\right]_\ess\right|\,\leq\,C\,\left[r_\veps\right]_\ess\,,
$$
which belongs to $L^\infty(\R_+;L^2)$ in view of \eqref{ub:dens-ess}. As for the residual part, we split it further into two parts: we have
\begin{align*}
\left|\frac{1}{\veps^{2\alpha}}\,\Pi\big(\rho_\veps,1\big)\,\mds{1}_{\{0\leq\rho_\veps<1/2\}}\right|\,\leq\,\frac{C}{\veps^{2\alpha}}\,\mds{1}_{\Omega_\res^\veps}\quad\mbox{ and }\quad
\left|\frac{1}{\veps^{2\alpha}}\,\Pi\big(\rho_\veps,1\big)\,\mds{1}_{\{\rho_\veps>2\}}\right|\,\leq\,\frac{C}{\veps^{2\alpha}}\,\left[\rho_\veps\right]_\res^\g\,.
\end{align*}
At this point, we can apply \eqref{ub:dens-res} to deduce the uniform boundedness of both terms in $L^\infty(\R_+;L^1)$.

This completes the proof of the lemma.
\end{proof}

We are now ready to state and prove the main result of the present subsection.
\begin{prop} \label{p:constr_a}
Fix $0<\alpha<1\leq\beta$.
Let $\bigl(\rho_\veps,u_\veps\bigr)_\veps$ be a sequence of weak solutions to system \eqref{eq:sing-NSC}-\eqref{eq:bc}, associated with  initial data $\bigl(\rho_{0,\veps},u_{0,\veps}\bigr)$ satisfying
the assumptions fixed in Section \ref{s:results}. Let $(r,u)$ be a limit point of the sequence $\big(r_\veps,u_\veps\big)_\veps$, as identified in Paragraph~{\ref{sss:ub-a}}.
Let $\theta$ be the quantity introduced in \eqref{cv:div-u}.

Then one has the following properties:
\begin{enumerate}[(i)]
 \item if $\beta>1$, then $\theta\equiv0$, $r\equiv0$ and $u\,=\,\big(u^h,0\big)$, where $u^h\,=\,u^h(t,x^h)$ is such that $\divh u^h\,=\,0$;
 \item in the case $\beta=1$, then $r$ and $u$ verify the same properties as above; moreover $\theta\,=\,\theta(t,x^h)$ and $u^h\,=\,-\,\nabla_h^\perp\theta$.
\end{enumerate}
\end{prop}

\begin{proof}
First of all, from \eqref{ub:div-u}, we immediately infer that
\begin{equation} \label{eq:div-u}
\div u\,=\,0\,. 
\end{equation}

Next, let us consider the momentum equation and focus for a while on the case $\beta>1$.
Since in this case the most singular term is of order $O(\veps^\beta)$, recall also Lemma \ref{l:p} above,
by testing the momentum equation against $\veps^\beta\,\psi$, where $\psi\in\mc D\big([0,T[\,\times\Omega;\R^3\big)$, and letting $\veps\ra0^+$, we easily gather
$\nabla\theta\,=\,0$.
Then $\theta=\theta(t)$, but it has to belong to $L^2(\Omega)$ for almost every time, hence $\theta\equiv0$.

In order to see the contribution of the rotation term, we project the equation onto its divergece-free component. Namely, we test the momentum equation
against $\veps\,\psi$, where $\psi\in\mc D\big([0,T[\,\times\Omega;\R^3\big)$ is a test function verifying $\div \psi=0$: by uniform bounds and \eqref{cv:rho-u}, it is straightforward to
get, in the limit $\veps\ra0^+$, that
$$
e^3\times u\,=\,\nabla\pi\,,
$$
for a suitable distribution $\pi\in L^2_T(\dot H^1)$. From this relation, it is a routine matter to
deduce that
$$
\d_3u^h\,=\,0\qquad\qquad\mbox{ and }\qquad\qquad \divh u^h\,=\,0\,.
$$
From the latter property and the fact that $\div u=0$, we immediately infer that $\d_3u^3=0$, which implies $u^3\equiv0$ in view of the boundary conditions \eqref{eq:bc}.

We now consider the mass equation, which we rewrite as
$$
\d_tr_\veps\,+\,\veps^{\beta-\alpha}\left(\veps^{-\beta}\div u_\veps\right)\,+\,\div\big(r_\veps\,u_\veps\big)\,=\,0
$$
in view of \eqref{cv:r_ess} and \eqref{cv:r_res}. Notice that the initial datum for this equation is $(r_\veps)_{|t=0}\,=\,\veps^{1-\alpha}\,r_{0,\veps}$, which obviously converges to $0$
when $\veps\ra0^+$, in view of the assumptions on the initial data. Using \eqref{cv:r_ess} and \eqref{cv:r_res} again, together with \eqref{cv:u}
and the constraints already established on $u$, we can pass to the limit in the previous equation to get
$$
\d_tr\,+\,\divh\big(r\,u^h\big)\,=\,0\,.
$$
Then, the limit quantity $r$ is transported by $u^h$, which is a divergece-free vector field; since the initial datum is $0$, we deduce that $r(t)=0$ for all times.

\medbreak
Take now $\beta=1$: in this case, the Coriolis term and the bulk viscosity term are singular at the same order. Recall that, by Lemma \ref{l:p}, the pressure term is of order $O(\veps^\alpha)$,
so it is of lower order. Hence, taking $\veps\,\psi$, with $\psi\in\mc D\big([0,T[\,\times\Omega\big)$, as a test function in the momentum equation and passing to the limit for $\veps\ra0^+$,
we find
$$
e^3\times u\,-\,\nabla\theta\,=\,0\,,
$$
which implies that $\theta\,=\,\theta(t,x^h)$, whence $u^h\,=\,u^h(t,x^h)$ and $u^h\,=\,-\nabla_h^\perp\theta$. Once these relations have been obtained, the rest of the analysis
follows the same steps as before.
\end{proof}

\begin{rem} \label{r:beta}
The same argument used in the previous proof actually shows that the limit is trivial whenever $0\leq\beta<1$.

Indeed, on the one hand relation \eqref{eq:div-u} still holds true, by \eqref{ub:div-u} when $\beta>0$, or by passing to the limit in the mass equation when $\beta=0$.
On the other hand, if $0\leq\beta<1$, the most singular term in the momentum equation
is the Coriolis term: we then infer that $u^h\equiv0$ in the limit. So, \eqref{eq:div-u} tells us that $\d_3u^3\equiv0$, which finally imples $u^3=0$ as well.
\end{rem}

The property $r\equiv0$ may look strange, but actually there is a deep reason for it, which will be apparent in Subsection \ref{ss:further}.

\subsubsection{The case $\alpha=0$} \label{sss:constr-0}
We now treat the case $\alpha=0$. Our first concern is to establish the convergence of the products $\rho_\veps\,u_\veps$, since, contrary to the previous paragraph, we have no more smallness on
$\rho_\veps-1$.
\begin{lemma} \label{l:rho-u}
Let $\alpha=0$ and $\g>1$ in \eqref{hyp:P}.
Let $\bigl(\rho_\veps,u_\veps\bigr)_\veps$ be a sequence of weak solutions to system \eqref{eq:sing-NSC_2D}, associated  with initial data~$\bigl(\rho_{0,\veps},u_{0,\veps}\bigr)$ satisfying
the assumptions fixed in Section~{\ref{s:results}}. With the same notation introduced in Paragraph \ref{sss:ub-0}, let $(r,u)$ be a limit point of the sequence
$\bigl(r_\veps,u_\veps\bigr)_\veps$.

Then the product $\big(r_\veps\,u_\veps\big)_\veps$ converges to $r\,u$ in the weak topology of $L^2_T(H_{\rm loc}^{-s_1-\de})$, for all $\de>0$ arbitrarily small (and such that $s_1+\de<1$).
In particular, the product $\bigl(\rho_\veps\,u_\veps\bigr)_\veps$ converges to $\rho\,u$ in the distributional sense.
\end{lemma}

\begin{proof}
Notice that, since $\g>1$, then $s_1$ in \eqref{ub:rho_e-trho} is always smaller than $1$. In addition, in view of Corollary \ref{c:product-2}, the product is continuous
from $H^{-\eta}\times H^1\ra H^{-\eta-\de}$ for any $0<\eta<1$ and $\de>0$ arbitrarily small: this property, together with
\eqref{est:u-L^2} and \eqref{ub:rho_e-trho} implies that $\big(r_\veps\,u_\veps\big)_\veps$ is uniformly bounded in $L^2_T(H^{-s_1-\de}_{\rm loc})$ for any $\de>0$ small and such that, in addition,
$s_1+\de<1$. On the other hand, taking $\eta$ close enough to $1$ in \eqref{cv:s_e}, we get that $\big(r_\veps\big)_\veps$ is strongly convergent
in $\mc C_T(H^{-s_1-\de}_{\rm loc})$, while $\big(u_\veps\big)_\veps$ is weakly convergent in $L^2_T(H^1)$: using again the continuity properties of the product
on those spaces yields the result.
\end{proof}

After the previous preliminary result, we can prove the analogous of Proposition \ref{p:constr_a}.
\begin{prop} \label{p:constr_0}
Let the space domain be $\R^2$. Set $\g>1$ in \eqref{hyp:P} and fix $\alpha=0$ and $\beta\geq1$.
Let $\bigl(\rho_\veps,u_\veps\bigr)_\veps$ be a sequence of weak solutions to system \eqref{eq:sing-NSC_2D}, associated with  initial data $\bigl(\rho_{0,\veps},u_{0,\veps}\bigr)$ satisfying
the assumptions fixed in Section \ref{s:results}. Let $(r,u)$ be a limit point of the sequence $\big(r_\veps,u_\veps\big)_\veps$, as identified in Paragraph~{\ref{sss:ub-0}}.
Let $\theta$ be the quantity introduced in \eqref{cv:div-u}.

Then the following properties hold true:
\begin{enumerate}[(i)]
 \item if $\beta>1$, then $\theta\equiv0$ and $r\equiv0$, while $u$ verifies $\div u\,=\,0$;
 \item in the case $\beta=1$, then $r$ and $u$ verify the same properties as above; moreover $\theta$ and $u$ are linked by the relation $u\,=\,-\,\nabla^\perp\theta$.
\end{enumerate}
\end{prop}

\begin{proof}
As it was the case when $\alpha>0$, the uniform bound \eqref{ub:div-u} implies again that $u$ is divergence-free. On the other hand, Lemma \ref{l:rho-u} allows us to pass to the limit in
the weak formulation of the mass equation: we get
$$
\d_tr\,+\,\div\big(r\,u\big)\,=\,0\,,\qquad\qquad\mbox{ with }\qquad r_{|t=0}\,=\,0\,.
$$
Equivalently, $\rho$ solves $\d_t\rho\,+\,\div\big(\rho\,u\big)\,=\,0$, with initial datum $\rho_{|t=0}\,=\,1$.
Since $\div u=0$, we infer that $\rho\equiv1$ for all times, i.e. $r(t)\equiv0$ for all $t\geq0$.

This having been established, the rest of the proof works exactly as in the case of Proposition \ref{p:constr_a}. We omit to give the details.
\end{proof}

\subsection{Further properties for the density oscillations} \label{ss:further}

The fact that the density oscillations $r_\veps$ completely disapper in the limit process, see Propositions \ref{p:constr_a} and \ref{p:constr_0} above, suggests that
the decomposition $\rho_\veps\,=\,1+\veps^\alpha r_\veps$ is maybe too rough. More precisely, the idea is that the perturbations of the reference state $\trho=1$ are of order higher than $\veps^\alpha$,
and seeing density variations in the limit requires to find the right order of those terms.

The goal of the present subsection is to show that this insight is indeed correct.

\paragraph{General considerations.}
For proving the previous claim, we start by defining
$$
V_\veps\,:=\,\rho_\veps\,u_\veps\,,\qquad \theta_\veps\,:=\,\frac{1}{\veps^\beta}\,\div u_\veps\qquad\mbox{ and }\qquad
f_\veps\,:=\,-\,\frac{1}{\veps^{2\alpha}}\,\nabla\Pi\big(\rho_\veps,1\big)\,+\,\mu\,\Delta u_\veps\,-\,\div\big(\rho_\veps u_\veps\otimes u_\veps\big)\,.
$$
Notice that, when proving \eqref{cv:rho-u}, we have already established that, for all $T>0$, one has
\begin{equation} \label{ub:V_e}
 \big(V_\veps\big)_\veps\,\subset\,L^2_T(L^2+L^{3/2}+L^{6\g/(6+\g)})\,\hookrightarrow\,L^2_T\big(L^2+H^{-1/2}+H^{-m}\big)\,,
\end{equation}
with $m\,=\,1\,-\,3/\g$.
Moreover, from \eqref{ub:u-Du}, \eqref{ub:div-u} and Lemma \ref{l:p}, we gather the uniform bounds
\begin{equation} \label{ub:f_e}
\big(\theta_\veps\big)_\veps\,\subset\,L^2_T(L^2)\qquad\qquad\mbox{ and }\qquad\qquad \big(f_\veps\big)_\veps\,\subset\,L^2_T(H^{-s})\quad \forall\;s>5/2\,,
\end{equation}
for all $T>0$ fixed. Remark that Lemma \ref{l:p} holds true up to the endpoint case $\alpha=0$ included, with no modifications in the proof. Finally, we introduce
$$
\sigma_\veps\,:=\,\frac{1}{\veps}\,\big(\rho_\veps\,-\,1\big)\,=\,\frac{1}{\veps^{1-\alpha}}\,r_\veps\,.
$$
We stress the fact that no uniform bounds are available, for the moment, for the sequence of $\sigma_\veps$'s.

With the previous notations, system \eqref{eq:sing-NSC} can be written as the following wave system:
\begin{equation} \label{eq:waves}
\left\{\begin{array}{l}
        \veps\,\d_t\s_\veps\,+\,\div V_\veps\,=\,0 \\[1ex]
        \veps^\beta\,\d_tV_\veps\,-\,\nabla\theta_\veps\,+\,\veps^{\beta-\alpha}\,P'(1)\,\nabla r_\veps\,+\,\veps^{\beta-1}\,e^3\times V_\veps\,=\,\veps^\beta\,f_\veps\,,
       \end{array}
\right.
\end{equation}
which has to be meant in the weak sense. Taking the $\curlh$ of the second equation and dividing by $\veps^{\beta-1}$ yields
\begin{equation} \label{eq:eta^3}
\veps\,\d_t\curlh V_\veps^h\,+\,\divh V_\veps^h\,=\,\veps\,\curlh f^h_\veps\,.
\end{equation}
At this point, we subtract the first equation in \eqref{eq:waves} from this latter relation and we compute the average with respect to $x^3$, to get
\begin{equation} \label{eq:curlV-sigma}
\d_t\big(\curlh\lan V^h_\veps\ran\,-\,\lan\s_\veps\ran\big)\,=\,\curlh\lan f^h_\veps\ran\,.
\end{equation}
This equation, together with the assumptions on the initial data and \eqref{ub:V_e}-\eqref{ub:f_e}, implies that
$$
\big(\lan\s_\veps\ran\big)_\veps\,\subset\,L^2_T(H^{-s-1})\qquad\forall\;T>0\,,\quad\forall\;s>5/2\,,
$$
and then there exists some $\s\,=\,\s(t,x^h)$ belonging to that space such that, up to an extraction, one has
\begin{equation} \label{cv:s}
\lan\s_\veps\ran\,\rightharpoonup\,\s\qquad\qquad\mbox{ in }\qquad L^2_T(H^{-s-1})
\end{equation}
for all $T>0$ and all $s>5/2$.
Let us point out that \eqref{eq:curlV-sigma} also gives compactness in time for the sequence $\big(\curlh\lan V_\veps^h\ran\,-\,\lan\s_\veps\ran\big)_\veps$.


Before going on, let us spend a few more words on the case $\alpha=0$.

\paragraph{The particular case when $\alpha=0$.}
It goes without saying that the previous argument holds true also when $\alpha=0$, with slight modifications. More precisely, since in that case the space domain is $\R^2$,
property \eqref{ub:f_e} as well as \eqref{cv:s} hold true for any $s>2$. Of course, there is no more need to take the vertical averages, so that \eqref{eq:curlV-sigma} and \eqref{cv:s}
are valid for the whole sequences of $\curl V_\veps$, $\s_\veps$ and $f_\veps$.

However, for later analysis (see Subsection \ref{ss:small} below) it is convenient to skimp on time integrability for the functions $\sigma_\veps$.
Indeed, we remark that property \eqref{ub:V-2D} implies that $\big(\curl V_\veps\big)_\veps$ is uniformly bounded in $L^\infty_T(H^{-s_2})$. Since $s_2\,=\,1/\g<1$,
from \eqref{eq:curlV-sigma} again we get that (recall that this time the space dimension is $2$)
\begin{equation} \label{ub:s_2D}
\big(\s_\veps\big)_\veps\,\subset\,L^\infty_T(H^{-s-1})\,,\qquad\qquad \forall\;T>0\,,\quad \forall\;s>2\,.
\end{equation}
Hence, there exists a distribution $\sigma\in L^\infty_{\rm loc}\big(\R_+;H^{-3-\de}\big)$, for all $\de>0$, such that, up to the extraction of a subsequence,
\begin{equation} \label{cv:s_2D}
\s_\veps\,\stackrel{*}{\rightharpoonup}\,\s\qquad\qquad\mbox{ in }\qquad L^\infty_T(H^{-3-\de})
\end{equation}
for all $T>0$ and all $\de>0$ arbitrarily small.

\section{Passing to the limit in the case $0<\alpha<1$} \label{s:limit-a}
We complete here the proof to Theorem \ref{t:alpha}, performing the limit in the weak formulation of equations \eqref{eq:sing-NSC} when $0<\alpha<1$
and $\beta\geq1$.

We have already seen in the proof to Proposition \ref{p:constr_a} how passing to the limit in the mass equation, and why this does not give any information on the limit dynamics. On the other hand,
the properties established in Subsection \ref{ss:further} are too rough to be able to prove convergence in the equation for $\s_\veps$: first of all,
we have uniform bounds only on their vertical averages, and moreover those bounds are in spaces which are too negative for giving sense to the product $\s_\veps\,u_\veps$ and take the limit
in that sequence.

Therefore, let us focus only on the momentum equation. Notice however that we will need to exploit the analysis of Subsection \ref{ss:further} in order to pass to
the limit in the Coriolis term.

\subsection{First convergence results} \label{ss:first-conv}
Let us consider a test function $\psi\in\mc D\big([0,T[\,\times\Omega;\R^3\big)$ such that $\psi\,=\,\big(\nabla^\perp_h\vphi,0\big)$, for some smooth and compactly supported
$\vphi\,=\,\vphi(t,x^h)$. We take the weak formulation of the momentum equation in \eqref{eq:sing-NSC} against such a $\psi$: we get
\begin{align} \label{eq:weak-for-limit}
\int^T_0\!\!\!\int_\Omega\left(-\rho_\veps u_\veps\cdot\d_t\psi-\rho_\veps u_\veps\otimes u_\veps:\nabla\psi+\dfrac{1}{\veps}e^3\times\rho_\veps u_\veps\cdot\psi+
\mu\nabla u_\veps:\nabla\psi\right)=
\int_\Omega\rho_{0,\veps}u_{0,\veps}\cdot\psi(0)\,,
\end{align}
due to the fact that $\div\psi=0$. Notice that, by hypotheses on the initial data and properties \eqref{cv:in-data}, we immediately gather
$$
\int_\Omega\rho_{0,\veps}u_{0,\veps}\cdot\psi(0)\,\longrightarrow\,\int_\Omega u_{0}\cdot\psi(0)\,=\,\int_{\R^2}\lan u^h_{0}\ran\cdot\psi^h(0)\,,
$$

On the other hand, the convergence of the viscosity term presents no difficulty, since it is linear in $u_\veps$. Moreover, the convergence 
of the $\d_t$ term follows from \eqref{cv:rho-u}: we get 
$$
-\int^T_0\!\!\!\int_\Omega\rho_\veps\,u_\veps\cdot\d_t\psi\,\longrightarrow\,-\int^T_0\!\!\!\int_{\R^2}u^h\cdot\d_t\psi^h\quad\mbox{ and }\quad
\int^T_0\!\!\!\int_\Omega\nabla u_\veps:\nabla\psi\,\longrightarrow\,\int_0^T\!\!\!\int_{\R^2}\nabla_h u^h:\nabla_h\psi^h\,.
$$

Finally, let us pass to the limit in the rotation term: recalling that $\psi\,=\,\big(\nabla^\perp_h\vphi,0\big)$,
by use of the mass equation, it is easy to obtain
\begin{align*}
\frac{1}{\veps}\int^T_0\!\!\!\int_\Omega e^3\times\rho_\veps\,u_\veps\cdot\psi\,&=\,\frac{1}{\veps}\int^T_0\!\!\!\int_\Omega\rho_\veps\,\big(u_\veps^h\big)^\perp\cdot\nabla^\perp_h\vphi\;=\;
\frac{1}{\veps}\int^T_0\!\!\!\int_{\R^2}\lan\rho_\veps\,u_\veps^h\ran\cdot\nabla\vphi \\
&=\,-\int_0^T\!\!\!\int_{\R^2}\lan\s_\veps\ran\,\d_t\vphi\,-\,\int_{\R^2}\lan r_{0,\veps}\ran\,\vphi\,.
\end{align*}
Hence, in view of the convergence properties \eqref{cv:in-data} and \eqref{cv:s}, one gathers
$$
\int^T_0\!\!\!\int_\Omega e^3\times\rho_\veps\,u_\veps\cdot\psi\,\longrightarrow\,-\int_0^T\!\!\!\int_{\R^2}\s\,\d_t\vphi\,-\,\int_{\R^2}\lan r_{0}\ran\,\vphi\,.
$$

\begin{rem} \label{r:vort-limit}
Notice that the previous argument, which seems to be necessary in order to take the limit of the rotation term, forces us to make the scalar function $\vphi$ appear
as a test function in the weak formulation of the limit equations. In other terms, we are obliged to consider the vorticity formulation of the limit dynamics.
\end{rem}

Therefore, in order to complete the passage to the limit, and then the proof of Theorem \ref{t:alpha}, it remains to us to prove the convergence of the convective term
$\rho_\veps u_\veps\otimes u_\veps$: this is the goal of the next subsection, where we resort to a compensated compactness argument,
combined with the decay estimates of Subsection \ref{app:heat}.

\subsection{The limit of the convective term} \label{ss:convective}

In this subsection, we show how taking the limit in the convective term. First of all, we reduce our problem to proving convergence in a convective term where the density function
is equal to $1$ and the velocity fields are smooth with respect to the space variable. Then, we apply a compensated compactness argument and exploit the system of wave equations \eqref{eq:waves}
in order to passing to the limit.

\subsubsection{Approximation and regularisation} \label{sss:app-reg}

The first step in passing to the limit in the convective term is the following approximation lemma.
\begin{lemma} \label{l:1_conv}
For any test function $\psi\,\in\,\mc{D}\bigl([0,T[\,\times\Omega;\R^3\bigr)$, one has
$$
\lim_{\veps\ra0^+}\left|\int^T_0\int_\Omega\rho_\veps u_\veps\otimes u_\veps:\nabla\psi\,dx\,dt\,-\,\int^T_0\int_\Omega u_\veps\otimes u_\veps:\nabla\psi\,dx\,dt\right|\,=\,0\,.
$$
\end{lemma}

The proof of relies on the fact that the difference of the two integrals is of order $O(\veps^\alpha)$: this is based on the uniform boundedness properties \eqref{ub:u-Du}, \eqref{ub:u},
\eqref{cv:r_ess} and \eqref{cv:r_res}. We omit to give the detailed argument here.

Next, it is convenient to introduce a regularisation of the velocity fields $u_\veps$. So, for any $M\in\N$, let us consider the low-frequency cut-off operator $S_M$ of a Littlewood-Paley decomposition,
as introduced in \eqref{eq:S_j} above.  For any $\veps>0$, we define
$$
u_{\veps,M}\,:=\,S_Mu_\veps\,,
$$
and analogous notation for all the other quantities here below. Observe that, in view of \eqref{ub:u-Du}, we can estimate
\begin{equation} \label{est:Id-S_M-u}
\left\|\big(\Id-S_M\big)u_\veps\right\|_{L^2_T(L^2)}\,\leq\,2^{-M}\,\left\|\nabla u_\veps\right\|_{L^2_T(L^2)}\,\leq\,C\,2^{-M}\,,
\end{equation}
for a constant $C>0$ independent of $\veps$. On the other hand, since $S_M$ is a bounded operator over all $H^s$ spaces, which moreover commutes with the space derivatives,
thanks to the uniform bounds of \eqref{ub:u} and \eqref{ub:div-u}, we get
\begin{equation} \label{est:u-e-M}
\left\|u_{\veps,M}\right\|_{L^2_T(H^s)}\,\leq\,C(T,s,M)\qquad\quad\mbox{ and }\qquad\quad
\left\|\div u_{\veps,M}\right\|_{L^2_T(H^s)}\,\leq\,\veps^\beta\,C(T,s,M)\,,
\end{equation}
for a constant depending only on the quantities in the brackets, but not on $\veps>0$.
We also notice that, thanks to relation \eqref{eq:LP-Sob}, Lemma \ref{l:Id-S} and Lebesgue dominated convergence theorem, we have the strong convergence
$$
S_Mu\,\longrightarrow\,u\qquad\qquad\mbox{ in }\qquad L^2_T(H^1)
$$
for $M\ra+\infty$, where $u$ is the vector-field identified in \eqref{cv:u}. This argument shows that, if we can pass to the limit in the convective term, where we have regularised
the velocity fields, then we can easily compute the limit when the appoximation parameter $M$ goes to $+\infty$.

The next lemma establishes that the errors created by the regularisation procedure are negligible, in the limit when $M\ra+\infty$.
\begin{lemma} \label{l:uxu_conv}
For any test function $\psi\,\in\,\mc{D}\bigl([0,T[\,\times\Omega;\R^3\bigr)$, one has
$$
\lim_{M\ra+\infty}\;\limsup_{\veps\ra0^+}\;
\left|\int^T_0\int_\Omega u_\veps\otimes u_\veps:\nabla\psi\,dx\,dt\,-\,\int^T_0\int_\Omega u_{\veps,M}\otimes u_{\veps,M}:\nabla\psi\,dx\,dt\right|\,=\,0\,.
$$
\end{lemma}

\begin{proof}
We start by writing the difference of the two integrals as
$$
\int^T_0\int_\Omega u_\veps\otimes u_\veps:\nabla\psi\,dx\,dt\,-\,\int^T_0\int_\Omega u_{\veps,M}\otimes u_{\veps,M}:\nabla\psi\,dx\,dt\,=\,I_1\,+\,I_2\,,
$$
where we have defined
$$
I_1\,:=\,\int^T_0\int_\Omega (\Id-S_M)u_{\veps}\otimes u_{\veps}:\nabla\psi\qquad\mbox{ and }\qquad I_2\,:=\,\int^T_0\int_\Omega u_{\veps,M}\otimes (\Id-S_M)u_{\veps}:\nabla\psi\,.
$$
At this point, from \eqref{ub:u-Du}, \eqref{est:Id-S_M-u} and \eqref{est:u-e-M} it is easy to deduce that
$$
\lim_{M\ra+\infty}\;\sup_{\veps>0}\;\left|I_j\right|\,=\,0\qquad\qquad\mbox{ for }\qquad j\,=\,1,2\,.
$$
This completes the proof of the statement.
\end{proof}

Before moving on, let us resort to the same notation introduced in Subsection \ref{ss:further} and set
$$
r_{\veps,M}\,:=\,S_Mr_\veps\,,\quad V_{\veps,M}\,:=\,S_MV_\veps\,,\quad \s_{\veps,M}\,:=\,S_M\s_\veps\,,\quad \theta_{\veps,M}\,:=\,S_M\theta_\veps\,,\quad
f_{\veps,M}\,:=\,S_Mf_\veps\,.
$$
Remark that applying operator $S_M$ to system \eqref{eq:waves} immediately yields
\begin{equation} \label{eq:waves_M}
\left\{\begin{array}{l}
        \veps\,\d_t\s_{\veps,M}\,+\,\div V_{\veps,M}\,=\,0 \\[1ex]
        \veps^\beta\,\d_tV_{\veps,M}\,-\,\nabla\theta_{\veps,M}\,+\,\veps^{\beta-\alpha}\,P'(1)\,\nabla r_{\veps,M}\,+\,\veps^{\beta-1}\,e^3\times V_{\veps,M}\,=\,\veps^\beta\,f_{\veps,M}\,.
       \end{array}
\right.
\end{equation}
Of course, the first equation can be also written as
$$
\veps^\alpha\,\d_tr_{\veps,M}\,+\,\div V_{\veps,M}\,=\,0\,.
$$
Notice that, thanks to \eqref{ub:f_e}, for all $T>0$ fixed and all $s\geq0$, one has
\begin{equation} \label{ub:f_e-M}
\sup_{\veps>0}\left\|f_{\veps,M}\right\|_{L^2_T(H^s)}\,\leq\,C(s,M,T)\,,
\end{equation}
where the positive constant $C(s,M,T)$  depends only on the quantities in the brackets.

It is apparent that we need to compare the two vector-fields $u_{\veps,M}$ and $V_{\veps,M}$: the next statement takes care of this.
\begin{lemma} \label{l:u-V}
For every $M\in\N$ and all $\veps\in\,]0,1]$, one has
$$
V_{\veps,M}\,=\,u_{\veps,M}\,+\,\veps^\alpha\,\mc V_{\veps,M}\,,
$$
where the sequence $\big(\mc V_{\veps,M}\big)_\veps$ verifies, uniformly in $\veps>0$, the bounds
$$
\left\|\mc V_{\veps,M}\right\|_{L^2_T(H^s)}\,\leq\,C(s,M,T)\,,
$$
for all $T>0$, $s\geq0$ and $M\in\N$ fixed.
\end{lemma}

\begin{proof}
The proof is straightforward: by definition, we decompose
$$
V_\veps\,=\,\rho_\veps\,u_\veps\,=\,u_\veps\,+\,\veps^\alpha\,r_\veps\,u_\veps\,.
$$
We set $\mc V_\veps\,:=\,r_\veps\,u_\veps$. By \eqref{ub:u} and \eqref{cv:r_ess}, we know that $\big(\left[\mc V_\veps\right]_\ess\big)_\veps$ is uniformly bounded in $L^2_T(L^{3/2})$, for all $T>0$.

Let us now focus on the residual part. If $\g<2$, estimate \eqref{ub:r_res} implies that $\big([r_\veps]_\res\big)_\veps$ is uniformly bounded in $L^\infty_T(L^\g)$, so that
$\big(\left[\mc V_\veps\right]_\res\big)_\veps$ is bounded in $L^2_T(L^p)$, with $p\,=\,6\g/(6+\g)\leq 2$. Therefore, by dual Sobolev embedding this term is uniformly bounded in some $L^2_T(H^{-m})$,
where $m=1\,-\,3/\g$.

Finally, suppose $\g\geq2$. In this case, the same computation as in \eqref{ub:r_res} shows that $\big([r_\veps]_\res\big)_\veps\subset L^\infty_T(L^2)$, hence also
$\big(\left[\mc V_\veps\right]_\res\big)_\veps$ is uniformly bounded in $L^2_T(L^{3/2})$.

This completes the proof of the lemma.
\end{proof}

From the previous lemma and bounds \eqref{est:u-e-M}, we immediately deduce the next statement, whose proof is hence omitted.
\begin{coroll} \label{c:u-V}
Let us define $\eta_\veps\,:=\,\curl V_\veps$ and $\omega_\veps\,:=\,\curl u_\veps$. For every $M\in\N$ and all $\veps\in\,]0,1]$, one has
the following properties, which hold for all $T>0$, $s\geq0$ and $M\in\N$ fixed:
\begin{align*}
&\eta_{\veps,M}\,=\,\omega_{\veps,M}\,+\,\veps^\alpha\,\z_{\veps,M}\,,\qquad\qquad\mbox{ with }\qquad \left\|\z_{\veps,M}\right\|_{L^2_T(H^s)}\,\leq\,C(s,M,T) \\
&\left\|\div V_{\veps,M}\right\|_{L^2_T(H^s)}\,\leq\,\veps^{\alpha}\,C(s,M,T)\,,
\end{align*}
for a suitable constant $C(s,M,T)>0$, depending only on the quantities in the brackets.
\end{coroll}



\subsubsection{Strong convergence of the vorticity} \label{sss:strong-vort}

As it will be apparent in the next paragraph, in order to pass to the limit we still need strong convergence properties for some quantity related to the velocity fields $V_\veps$,
namely for the vertical averages of the vorticity functions.
We have already remarked in Subsection \ref{ss:further} that the sequence
$\big(\lan\eta^3_\veps\ran\,-\,\lan\sigma_\veps\ran\big)_\veps$ is compact in suitable spaces, but this information is not enough, since those quantities
are not compact \textsl{a priori}, when considered separately.

In this part, we are going to show that actually $\big(\lan\eta^3_{\veps,M}\ran\big)_\veps$ is compact in appropriate spaces. Such a property cannot
really come from the wave system \eqref{eq:waves}, due to the anisotropy of scaling: notice that acoustic waves propagate at speed $\veps^{-\alpha}$, so that
the dispersive estimates of e.g. \cite{F-G-GV-N}, \cite{F-N_CPDE} are out of use here.

The key observation to get compactness, instead, is that the second equation in \eqref{eq:waves} hides a heat-like equation for $\div V_\veps$, with fast oscillations in time.
Hence, from that equation we can derive strong decay for $\div V_\veps$ and its higher order derivatives. Then, the idea is to use this decay in the equation for $\lan \eta^3_\veps\ran$,
see \eqref{eq:eta^3}, to get the compactness in time of higher order derivatives of that quantity.

\paragraph{Decay of higher order derivatives of $\div_h\lan V_{\veps,M}^h\ran$.}
In order to fully justify the previous heuristics, let us proceed in the following way. First of all, we introduce the operators $\P$ to be the Leray-Helmholtz
projector onto the divergence-free vector-fields and $\Q$ to be the projector orthogonal to $\P$ (with respect to the $L^2$ scalar product). Then, resorting to the notation introduced
in \eqref{dec:vert-av}, we can decompose
$$
\lan V^h_{\veps,M}\ran\,=\,\P[\lan V^h_{\veps,M}\ran]\,+\,\Q[\lan V^h_{\veps,M}\ran]\,,\qquad\qquad\mbox{ with }\qquad \Q[\lan V^h_{\veps,M}\ran]\,:=\,\nabla_h\lan\Phi_{\veps,M}\ran\,.
$$
Notice that
$$
\P[\lan V^h_{\veps,M}\ran]\,=\,-\nabla_h^\perp(-\Delta_h)^{-1}\lan\eta^3_{\veps,M}\ran\qquad\quad\mbox{ and }\qquad\quad \divh\lan V^h_{\veps,M}\ran\,=\,\Delta_h\lan\Phi_{\veps,M}\ran\,.
$$

Next, consider the smoothed wave system \eqref{eq:waves_M}. After recalling that $\theta_\veps\,=\,\veps^{-\beta}\,\div u_\veps$, in view of Lemma \ref{l:u-V} we can write
the second equation as
$$
\veps^{2\beta}\,\d_tV_{\veps,M}\,-\,\nabla\div V_{\veps,M}\,+\,\veps^{2\beta-\alpha}\,P'(1)\,\nabla r_{\veps,M}\,+\,\veps^{2\beta-1}\,e^3\times V_{\veps,M}\,=\,\veps^{2\beta}\,f_{\veps,M}\,-\,
\veps^\alpha\,\nabla\div\mc V_{\veps,M}\,.
$$
Therefore, taking the vertical averages of the horizontal components yields an equation for $\lan\Phi_{\veps,M}\ran$:
\begin{equation} \label{eq:Phi}
\d_t\lan\Phi_{\veps,M}\ran\,-\,\frac{1}{\veps^{2\beta}}\,\Delta_h\lan\Phi_{\veps,M}\ran\,=\,\lan G_{\veps,M}\ran\,,
\end{equation}
where we have defined 
\begin{align} \label{def:g}
\lan G_{\veps,M}\ran\,&:=\,-\,\big(-\Delta_h\big)^{-1}\div_h\lan f^h_{\veps,M}\ran\,-\,\veps^{\alpha-2\beta}\,\div_h\lan\mc V_{\veps,M}\ran\, \\
&\qquad\qquad\qquad +\,\veps^{-1}\,\big(-\Delta_h\big)^{-1}\div_h\lan V_{\veps,M}^{h}\ran^\perp\,-\,\veps^{1-2\alpha}\,P'(1)\,\lan\s_{\veps,M}\ran\,. \nonumber
\end{align}

Let us fix some $s\geq1$, whose precise value will be decided later. By Theorem \ref{t:fast-heat}, applied with $\nu(\veps)\,=\,\veps^{2\beta}$
(and to $\nabla\Phi_{0,\veps}$ and $\nabla G_\veps$ to avoid the singularity of the operator $(-\Delta_h)^{-1}$), for any $T>0$ and any $0<\de<1$ we get,
for all $\veps\leq\veps_0(s,\de)$,
\begin{align*}
\left\|(-\Delta_h)^s\,\nabla_h\lan\Phi_{\veps,M}\ran\right\|_{L^2\big(\,]\de,T[\,;L^2\big)}\,\leq\,C\,\veps^{s\beta}\,\left(\left\|\nabla_h\lan\Phi_{0,\veps,M}\ran\right\|_{L^2}\,+\,
\left\|\nabla_h\lan G_{\veps,M}\ran\right\|_{L^2_T(H^s)}\right)\,,
\end{align*}
for a positive constant $C$ just depending on $T$, $s$ and $\de$. At this point, it is easy to see that $\left\|\nabla_h\lan\Phi_{0,\veps,M}\ran\right\|_{L^2}\,\leq\,C$,
for some constant which does not depend on $\veps$, nor on $M$ and on the various parameters $T$, $s$ and $\de$. On the other hand, by definition
\eqref{def:g}, the uniform bounds established in Subsections \ref{ss:unif-bounds} and \ref{ss:further} and Bernstein's inequalities, we also have
\begin{align*}
\left\|\nabla_h\lan G_\veps\ran\right\|_{L^2_T(H^s)}\,\leq\,C\,\left(1\,+\,\veps^{\alpha-2\beta}\,+\,\veps^{-1}\,+\,\veps^{1-2\alpha}\right)\,,
\end{align*}
for a positive constant $C\,=\,C(s,M,T)$ depending only on the quantities on the brackets, but not on $\veps>0$.
Notice that the worst exponent is $\alpha-2\beta$: therefore, taking
\begin{equation} \label{def:s}
s\,\geq\,s_0\,:=\,2\,+\,\frac{1\,-\,\alpha}{\beta}\,,
\end{equation}
we finally deduce that, for all $T>0$ and all $\de\in\,]0,1[\,$, one has
\begin{equation} \label{est:Phi}
\left\|(-\Delta_h)^s\,\nabla_h\lan\Phi_{\veps,M}\ran\right\|_{L^2\big(\,]\de,T[\,;L^2\big)}\,\leq\,C\,\veps^{\beta(s-2)+\alpha}\,\leq\,C\,\veps\,,
\end{equation}
for a constant $C\,=\,C(s,M,T,\de)$ depending only on the quantity in the brackets, but uniform in $\veps\,\in\,]0,\veps_0(s,\de)[\,$, where the parameter
$\veps_0(s,\de)$ is the one given by Theorem \ref{t:fast-heat}.

\paragraph{Compactness of the averaged vorticity.}
Now we are ready to derive compactness properties for the vertical averages of the vorticity functions.
We explicitly point out that, in the next argument, $M$ plays the role of a fixed parameter: all the compactness properties are with respect to $\veps$, working at $M$ fixed.

We start by observing that, from the second equation in \eqref{eq:waves_M}, we gather (recall equation \eqref{eq:eta^3} above)
$$
\d_t\lan\eta^3_{\veps,M}\ran\,=\,\curlh\lan f_{\veps,M}^h\ran\,-\,\frac{1}{\veps}\,\divh\lan V_{\veps,M}^h\ran
$$
for all $\veps>0$ and all $M\in\N$. Recall that $\divh\lan V^h_{\veps,M}\ran\,=\,\Delta_h\lan\Phi_{\veps,M}\ran$; hence, after applying the operator
$(-\Delta_h)^{s_0}$, where $s_0$ has been fixed in \eqref{def:s}, we get
$$
\d_t(-\Delta_h)^{s_0}\lan\eta^3_{\veps,M}\ran\,=\,(-\Delta_h)^{s_0}\curlh\lan f_{\veps,M}^h\ran\,+\,\frac{1}{\veps}\,(-\Delta_h)^{s_0+1}\lan \Phi_{\veps,M}^h\ran\,.
$$

From the previous equation, thanks to estimate \eqref{est:Phi}, we derive the following property: with the notation of Theorem \ref{t:fast-heat},
for all $\de\in\,]0,1[\,$, define $\veps_\de\,:=\,\veps_0(s_0,\de)$; then, for any $M\in\N$ fixed, the sequence
$$
\left(\d_t\lan\eta^3_{\veps,M}\ran\right)_{\veps\leq\veps_\de}\qquad \mbox{ is uniformly bounded in }\quad L^2\big(\,]\de,T[\,;\dot H^{s_0}(\R^2)\big)\,.
$$
Observe that $s_0\geq2$; therefore, up to working with $s_0+\eta$ (for $\eta>0$ small) in the case when $s_0-1\in\N$,
by Sobolev embeddings in H\"older spaces (see e.g. Theorem 1.50 of \cite{B-C-D}) we gather that
$\left(\lan\eta^3_{\veps,M}\ran\right)_{\veps\leq\veps_\de}$ is uniformly bounded in e.g. the space $W^{1,2}\big(\,]\de,T[\,;\mc C^{0,\mf s}(\R^2)\big)$, for some $\mf s\in\,]0,1[\,$,
hence in $W^{1,2}\big(\,]\de,T[\,;L^2_{\rm loc}(\R^2)\big)$.
On the other hand, by uniform bounds (see \eqref{ub:V_e} for instance), we know that $\left(\lan\eta^3_{\veps,M}\ran\right)_{\veps\leq\veps_\de}$
is uniformly bounded in the space $L^2\big(\,]\de,T[\,;H^1_{\rm loc}(\R^2)\big)$.
Therefore, an application of Aubin-Lions lemma implies that, up to the extraction of a suitable subsequence,
\begin{equation} \label{cv:vort}
\left(\lan\eta^3_{\veps,M}\ran\right)_{\veps\leq\veps_\de}\qquad \mbox{ is strongly convergent in }\quad L^2\big(\,]\de,T[\,;L^2_{\rm loc}(\R^2)\big)\,.
\end{equation}
As a consequence, in view of Corollary \ref{c:u-V}, one has
$$
\lan\eta^3_{\veps,M}\ran\,\longrightarrow\,\omega_M\qquad\qquad \mbox{ in }\qquad L^2\big(\,]\de,T[\,;L^2(K)\big)\,,
$$
for any $K\subset\Omega$ compact, in the limit when $\veps\ra0^+$, at any $M\in\N$ fixed. We recall that we have denoted $\omega\,:=\,\curlh u^h$, where $u$ is the limit velocity field,
and $\omega_M\,:=\,S_M\omega$.

\subsubsection{The compensated compactness argument} \label{sss:comp-comp}

In light of Lemmas \ref{l:1_conv} and \ref{l:uxu_conv}, we have reduced our problem to passing to the limit in the integral
\begin{align*}
-\int^T_0\!\!\!\int_\Omega u_{\veps,M}\otimes u_{\veps,M}:\nabla\psi\,&=\,\int^T_0\!\!\!\int_\Omega\div\bigl(u_{\veps,M}\otimes u_{\veps,M}\bigr)\cdot\psi\,=\,
\int^T_0\!\!\!\int_{\R^2}\left(\mc T^1_{\veps,M}\,+\,\mc{T}^2_{\veps,M}\right)\cdot\psi^h\,,
\end{align*}
where we have defined
\begin{align*}
\mc T^1_{\veps,M}\,:=\,\divh\bigl(\lan u^h_{\veps,M}\ran\otimes\lan u^h_{\veps,M}\ran\bigr)\qquad\qquad\mbox{ and }\qquad\qquad
\mc T^2_{\veps,M}\,:=\,\divh\lan\wtilde u^h_{\veps,M}\otimes\wtilde u^h_{\veps,M}\ran\,.
\end{align*}
Notice that the integration by parts in the previous equality is fully justified, since now each term is smooth in the space variable. Moreover, we have used the structure of the test function $\psi$,
whose third component is identically zero and whose horizontal components depend only on $x^h$.

We deal separately with each one of the previous terms here below. In the argument that follows, we will denote by $\mc R_{\veps,M}$  any remainder term, i.e. any term
having the property that
\begin{equation} \label{eq:remainder}
\lim_{M\ra+\infty}\,\limsup_{\veps\ra0}\,\left|\int^T_0\int_\Omega \mc{R}_{\veps,M}\,\cdot\,\psi\,dx\,dt\right|\,=\,0
\end{equation}
for all test functions $\psi\,\in\,\mc{D}\bigl([0,T[\,\times\Omega;\R^3\bigr)$ such that $\psi\,=\,\big(\psi^h,0\big)$, with $\psi^h\,=\,\psi^h(t,x^h)$ which satisfies moreover $\divh\psi^h=0$.

\paragraph{The $\mc T^1_{\veps,M}$ term.}

We start by considering the term $\mc T^1_{\veps,M}$: standard computations yield
\begin{align}
\mc T^1_{\veps,M}\,&=\,\divh\bigl(\lan u^h_{\veps,M}\ran\otimes\lan u^h_{\veps,M}\ran\bigr)\,=\,
\divh\lan u^h_{\veps,M}\ran\,\lan u^h_{\veps,M}\ran\,+\,\lan u^h_{\veps,M}\ran\cdot\nabla_h\lan u^h_{\veps,M}\ran \label{eq:T^1} \\
&=\divh\lan u^h_{\veps,M}\ran\,\lan u^h_{\veps,M}\ran\,+\,\frac{1}{2}\,\nabla_h\left|\lan u^h_{\veps,M}\ran\right|^2\,+\,\lan\omega_{\veps,M}^3\ran\,\lan u^h_{\veps,M}\ran^\perp\,. \nonumber
\end{align}
Let us recall that $\omega^3_\veps\,=\,\curlh u_\veps^h$ and $\eta^3_\veps\,=\,\curlh V_\veps^h$, and analogous formulas for the regularised
quantities $\omega^3_{\veps,M}$ and $\eta^3_{\veps,M}$.

Notice that the first two terms of the last relation contribute as remainders, in the sense of relation \eqref{eq:remainder},
in view of the uniform bounds stated in \eqref{est:u-e-M}. On the other hand, thanks to Corollary \ref{c:u-V} we can write
$$
\mc T^1_{\veps,M}\,=\,\mc R_{\veps,M}\,+\,\lan\eta_{\veps,M}^3\ran\,\lan u^h_{\veps,M}\ran^\perp\,.
$$
In order to understand the limit of the $\mc T^1_{\veps,M}$ term, we have then to take the limit (in the sense of distribution) of the last term in the right-hand side of the previous relation.

To this end, let us consider the limit for $\veps\ra0^+$, at any $M\in\N$ fixed, of the integral
$$
\int^T_0\int_{\R^2}\lan\eta_{\veps,M}^3\ran\,\lan u^h_{\veps,M}\ran^\perp\cdot\psi^h\,dx^h\,dt\,.
$$
Let us fix some $0<\de<1$. In view of \eqref{cv:vort} and Sobolev embeddings, up to an extraction, we know that
$\left(\lan\eta^3_{\veps,M}\ran\right)_{\veps\leq\veps_\de}$ strongly converges to $\omega_M$ in $L^2\big(\,]\de,T[\,;L^2(K)\big)$, where we have denoted by $K$ the support (in the space variable)
of the test function $\psi$.
Combining this property with the uniform boundedness of $\big(\lan u^h_{\veps,M}\ran\big)_\veps$ in e.g. $L^2_T(L^2)$,
we deduce that
$$
\int^T_\de\int_{\R^2}\lan\eta_{\veps,M}^3\ran\,\lan u^h_{\veps,M}\ran^\perp\cdot\psi^h\,dx^h\,dt\,\longrightarrow\,
\int^T_\de\int_{\R^2}\omega_{M}\,u^{h,\perp}_{M}\cdot\psi^h\,dx^h\,dt\,.
$$
On the other hand, due to the uniform boundedness of both $\big(\lan u^h_{\veps,M}\ran\big)_\veps$ and $\big(\lan\eta^3_{\veps,M}\ran\big)_\veps$ in $L^2_T(L^2)$,
we have
$$
\left|\int^\de_0\int_{\R^2}\lan\eta_{\veps,M}^3\ran\,\lan u^h_{\veps,M}\ran^\perp\cdot\psi^h\,dx^h\,dt\right|\,\leq\,C_\de\,,
$$
where $C_\de\,\longrightarrow\,0$ when $\de\ra0^+$.
Putting those properties together finally shows that, in the limit $\veps\ra0^+$, one has
$$
\int^T_0\int_{\R^2}\lan\eta_{\veps,M}^3\ran\,\lan u^h_{\veps,M}\ran^\perp\cdot\psi^h\,\longrightarrow\,\int^T_0\int_{\R^2}\omega_{M}\,u_{M}^{h,\perp}\cdot\psi^h\,=\,
-\int^T_0\int_{\R^2}u_M\otimes u_M:\nabla_h\psi^h\,,
$$
where the last equality holds, since we can perform the computations in \eqref{eq:T^1} backwards.

Therefore, after taking the limit for $M\ra+\infty$ and arguing like in the proof to Lemma \ref{l:uxu_conv}, we have finally proved that
\begin{equation} \label{eq:limit-T^1}
\lim_{M\ra+\infty}\,\limsup_{\veps\ra0^+}\int^T_0\int_{\R^2} \mc T^1_{\veps,M}\cdot\psi^h\,dx\,dt\,=\,-\int^T_0\int_{\R^2}u^h\otimes u^h:\nabla_h\psi^h\,dx\,dt
\end{equation}
for all test function $\psi\,=\,\big(\nabla_h\vphi,0\big)$, with $\vphi\in\mc D\big([0,T[\,\times\R^2\big)$. Recall that $u=\big(u^h,0\big)$ is the limit velocity field
identified in \eqref{cv:u}, which verifies the properties established in Proposition \ref{p:constr_a}.

\paragraph{The $\mc T^2_{\veps,M}$ term.}
Let us now consider the term $\mc{T}^2_{\veps,M}$: exactly as done above, and in view of Lemma \ref{l:u-V} and Corollary \ref{c:u-V}, we can write
\begin{align*}
\mc{T}^2_{\veps,M}\,&=\,\divh\left(\lan \wtilde{u}^h_{\veps,M}\otimes \wtilde{u}^h_{\veps,M}\ran\right)
\,=\,\lan \divh\bigl(\wtilde{u}^h_{\veps,M}\bigr)\;\wtilde{u}^h_{\veps,M}\ran\,+\,
\dfrac{1}{2}\,\lan\nabla_h\left|\wtilde{u}^h_{\veps,M}\right|^2\ran\,+\,
\lan \wtilde{\omega}^3_{\veps,M}\;\left(\wtilde{u}^h_{\veps,M}\right)^\perp\ran \\
&=\,\mc R_{\veps,M}\,+\,\lan \divh\bigl(\wtilde{V}^h_{\veps,M}\bigr)\;\wtilde{V}^h_{\veps,M}\ran\,+\,
\lan \wtilde{\eta}^3_{\veps,M}\;\left(\wtilde{V}^h_{\veps,M}\right)^\perp\ran
\end{align*}

Let us focus on the last term for a while: with the notations introduced in \eqref{dec:vert-av}, we have
\begin{align*}
&\left({\curl}\wtilde{V}_{\veps,M}\right)^h\,=\,\d_3\wtilde{W}^h_{\veps,M}\qquad\qquad\qquad\mbox{ with }\qquad
\wtilde{W}^h_{\veps,M}\,:=\,\left(\wtilde{V}^h_{\veps,M}\right)^\perp\,-\,
\d_3^{-1}\nabla^\perp_h\wtilde{V}^3_{\veps,M}\,;\\
&\left({\curl}\wtilde{V}_{\veps,M}\right)^3\,=\,{\curlh}\wtilde{V}^h_{\veps,M}\,=\,\wtilde{\eta}^3_{\veps,M}\,.
\end{align*}
From the momentum equation in \eqref{eq:waves_M}, where we take the mean-free part and then apply the $\curl$ operator, we immediately infer an equation for those quantities:
\begin{equation} \label{eq:mean-free}
\begin{cases}
\veps\,\d_t\wtilde{W}^h_{\veps,M}\,-\,\wtilde{V}^h_{\veps,M}\,=\,\veps\,\d_3^{-1}\left(\curl\wtilde{f}_{\veps,M}\right)^h \\[1ex]
\veps\,\d_t\wtilde{\eta}^3_{\veps,M}\,+\,\divh\wtilde{V}^h_{\veps,M}\,=\,\veps\,\curlh\wtilde{f}^h_{\veps,M}\,.
\end{cases}
\end{equation}
Thanks to the previous relations and the bounds in \eqref{ub:f_e-M}, we can write
\begin{align*}
\wtilde{\eta}^3_{\veps,M}\,\left(\wtilde{V}^h_{\veps,M}\right)^\perp\,&=\,
\veps\,\d_t\left(\wtilde{W}^h_{\veps,M}\right)^\perp\,\wtilde{\eta}^3_{\veps,M}\,-\,\veps\,\wtilde{\eta}^3_{\veps,M}\,\d_3^{-1}\left(\curl\wtilde{f}_{\veps,M}\right)^{h,\perp} \\
&=\,-\,\veps\,\left(\wtilde{W}^h_{\veps,M}\right)^\perp\,\d_t\wtilde{\eta}^3_{\veps,M}\,+\,\mc{R}_{\veps,M}\;=\;
\left(\wtilde{W}^h_{\veps,M}\right)^\perp\,\divh\wtilde{V}^h_{\veps,M}\,+\,\mc{R}_{\veps,M}\,.
\end{align*}

Therefore, we finally arrive at the expression
\begin{align*}
\mc{T}^2_{\veps,M}\,&=\,\lan \div_{\!h}\bigl(\wtilde{V}^h_{\veps,M}\bigr)\,
\left(\wtilde{V}^h_{\veps,M}\,+\,\left(\wtilde{W}^h_{\veps,M}\right)^\perp\right)\ran\,+\,\mc{R}_{\veps,M} \\
&=\,\lan \div\wtilde{V}_{\veps,M}\,\left(\wtilde{V}^h_{\veps,M}\,+\,\left(\wtilde{W}^h_{\veps,M}\right)^\perp\right)\ran\,-\,
\lan \d_3\wtilde{V}^3_{\veps,M}\,\left(\wtilde{V}^h_{\veps,M}\,+\,\left(\wtilde{W}^h_{\veps,M}\right)^\perp\right)\ran\,+\,\mc R_{\veps,M}\,.
\end{align*}
Notice that the first term in the right-hand side of the last equality is a remainder, in the sense of \eqref{eq:remainder}, in view of Corollary \ref{c:u-V}.
Concerning the second term in the right-hand side, instead, we use the definition of $W_{\veps,M}$: direct computations show that
\begin{eqnarray*}
\d_3\wtilde{V}^3_{\veps,M}\left(\wtilde{V}^h_{\veps,M}+\left(\wtilde{W}^h_{\veps,M}\right)^\perp\right) & = & 
\d_3\!\!\left(\wtilde{V}^3_{\veps,M}\left(\wtilde{V}^h_{\veps,M}+\left(\wtilde{W}^h_{\veps,M}\right)^\perp\right)\right)-
\wtilde{V}^3_{\veps,M}\,\d_3\!\left(\wtilde{V}^h_{\veps,M}+\left(\wtilde{W}^h_{\veps,M}\right)^\perp\right) \\
& = & \mc{R}_{\veps,M}\,-\,\frac{1}{2}\,\nabla_h\left|\wtilde{V}^3_{\veps,M}\right|^2\;=\;\mc{R}_{\veps,M}\,.
\end{eqnarray*}

In the end, we have just proved that $\mc{T}^2_{\veps,M}\,=\,\mc R_{\veps,M}$, hence the contribution of this term vanishes in the limit:
\begin{equation} \label{eq:limit-T^2}
\lim_{M\ra+\infty}\,\limsup_{\veps\ra0^+}\int^T_0\int_{\R^2} \mc T^2_{\veps,M}\cdot\psi^h\,=\,0\,.
\end{equation}

\subsection{Identifying the limit equation} \label{ss:limit-eq}
Let us resume the proof of convergence in relation \eqref{eq:weak-for-limit}.
All the terms appearing therein have already been treated in Subsection \ref{ss:first-conv}, except the convective term. As for the latter, in view of \eqref{eq:limit-T^1} and
\eqref{eq:limit-T^2}, Lemmas \ref{l:1_conv} and \ref{l:uxu_conv} imply that
\begin{align*}
\lim_{\veps\ra0^+}-\int^T_0\int_{\Omega}\rho_\veps\,u_\veps\otimes u_\veps:\psi\,dx\,dt\,&=\,-\int^T_0\int_{\R^2}u^h\otimes u^h:\nabla_h\psi^h\,dx\,dt\,.
\end{align*}

Now recall that $\psi^h\,=\,\nabla_h^\perp\vphi$, where $\vphi\,=\,\vphi(t,x^h)$. As stated in Remark \ref{r:vort-limit}, our approach for dealing with the Coriolis term forces us
to consider the vorticity formulation of the limit dynamics, when tested against the test function $\vphi$.
Hence, we still have to make an integration by parts with respect to the operator $\nabla_h^\perp$, which amounts exactly to apply the vorticity operator to the equations.

Straightforward computatons show that
\begin{align*}
\int_{\R^2}\lan u^h_{0}\ran\cdot\psi^h(0)\,&=\,-\int_{\R^2}\curlh\lan u^h_{0}\ran\,\vphi(0) \\
-\int^T_0\!\!\!\int_{\R^2}u^h\cdot\d_t\psi^h\,&=\,\int^T_0\!\!\!\int_{\R^2}\omega\,\d_t\vphi \\
\mu\int_0^T\!\!\!\int_{\R^2}\nabla_h u^h:\nabla_h\psi^h\,&=\,-\,\mu\int_0^T\!\!\!\int_{\R^2}\nabla_h\omega\cdot\nabla_h\vphi \\
-\int^T_0\int_{\R^2}u^h\otimes u^h:\nabla_h\psi^h\, &=\,\int^T_0\int_{\R^2}\omega\,u^h\cdot\nabla_h\vphi\,.
\end{align*}
Therefore, the final expression coincides exactly with the weak formulation of equation \eqref{eq:limit-a}, up to multiplication by the factor $-1$.

In the end, we have completed the proof to Theorem \ref{t:alpha}.

\medbreak
Before concluding this part, a couple of remarks are in order.

\begin{rem} \label{r:sigma}
Notice that the limit dynamics is somehow underdetermined, since we have only one equation, namely  \eqref{eq:limit-a}, for the two unknowns $\omega$ and $\sigma$.

Very likely, $\s$ solves a transport equation by $u^h$; nonetheless, we are not able to prove rigorously that this is actually the case. Indeed, first of all
we have strong convergence properties only on the vertical means of $u_{\veps,M}$, and, even more importantly, we have uniform bounds only for $\lan\sigma_{\veps}\ran$,
whereas we know nothing on the oscillating component $\wtilde\sigma_\veps$. Those facts represent a true obstacle in taking the limit in the mass equation (divided by $\veps$)
and finding an equation for the limit density profile $\sigma$.
\end{rem}

\begin{rem} \label{r:limit-eq}
Observe that, when $\beta=1$, from Proposition \ref{p:constr_a} we get that $\omega=\curlh u^h=-\Delta_h\theta$. Hence, an equivalent formulation of equation \eqref{eq:limit-a} is
$$
\d_t\big(-\Delta_h\theta-\s\big)\,+\,\nabla^\perp_h\theta\cdot\nabla_h\Delta_h\theta\,+\,\mu\,\Delta^2\theta\,=\,0\,,
$$
which may look more familiar to the reader (see e.g. the limit equations in \cite{F-G-N}, \cite{F_MA}).
\end{rem}

\section{Convergence in the case $\alpha=0$} \label{s:limit-0}

Let us now take $\alpha=0$ and pass to the limit in this case. Recall that, now, the equations are set in the $2$-dimensional domain $\R^2$, and we have taken $\g>1$ in \eqref{hyp:P}.

The first important step is to make the smallness of the functions $r_\veps$'s quantitative:
this is also the essential reason for our restricting to a $2$-D domain when $\alpha=0$.
After that, the rest of the convergence proof is pretty similar to the arguments exposed above for the case $\alpha>0$. For this reason, we will only sketch
those arguments.

The last part of this section is devoted to a conditional convergence result (in the same spirit of Theorem 5.8 of \cite{F-G_RMI}),
where we are able to identify a complete system of equations describing the limit dynamics.
However, the result is only conditional: for getting it, we need to impose \textsl{a priori} higher order \emph{uniform bounds} for the family of velocity fields
(and additional regularity on the initial densities), which seem to be hardly satisfied.

\subsection{Smallness of the density fluctuations} \label{ss:small}

When $\alpha=0$, the smallness of the density fluctuation functions $r_\veps$ does not come from the smallness of the Mach number, which is now of order $1$.
In order to get that property, we resort to an interpolation argument, in the same spirit of the one used in \cite{F-G_RMI}.

\begin{prop} \label{p:s-uniform}
There exist $0<\wtilde s<1$ and $0<\k<1$ such that the uniform embeddings
$$ 
\left(\frac{1}{\veps^\k}\;r_\veps\right)_\veps\;\subset\;L^\infty\bigl([0,T];H^{-\wtilde s}(\Omega)\bigr)\qquad\mbox{ and }\qquad
\left(\frac{1}{\veps^\k}\;r_\veps\,u_\veps\right)_\veps\;\subset\;L^2\bigl([0,T];H^{-\wtilde s-\delta}(\Omega)\bigr)
$$ 
hold true for any $T>0$ and all $\delta>0$ arbitrarily small.
\end{prop}

\begin{proof}
Recall that, by \eqref{ub:rho_e-trho}, the sequence $\big(r_\veps\big)_\veps$ is uniformly bounded in $L^\infty_T(H^{-s_1})$ for all $T>0$, where $s_1<1$. On the other hand,
by \eqref{ub:s_2D}, we know that $\big(\s_\veps\big)_\veps$ is uniformly bounded in $L^\infty_T(H^{-s})$, for any $s>3$.

Keeping in mind that $\s_\veps\,=\,r_\veps/\veps$, fixed an $s>3$, an interpolation of the previous uniform bounds yields
$$
\frac{1}{\veps^\kappa}\|r_\veps\|_{L^\infty_T(H^{-a})}\,\leq\,C\,\|r_\veps\|_{L^\infty_T(H^{-s_1})}^{1-\k}\;\left(\frac{1}{\veps}\,\|r_\veps\|_{L^\infty_T(H^{-s})}\right)^{\!\!\k}\,,
$$
under the condition that $a\,=\,a(\k)\,=\,(1-\k)\,s_1\,+\,\k\,s$, for some $\k\in\,]0,1[\,$. Taking $\k$ small enough and setting $\wtilde s\,=\,a(\k)$ entails the former claimed bound.

As for the latter uniform bound, it is a straightforward consequence of the previous one, of the property $\big(u_\veps\big)_\veps\,\subset\,L^2_T(H^1)$ and point (ii) in Corollary \ref{c:product-2}.
\end{proof}

Thanks to the previous result, we can establish the equivalent of Lemma \ref{l:u-V} and Corollary \ref{c:u-V}. Remark that, here, we have not regularised any quantity yet.
The proof is straightforward, hence omitted.
\begin{coroll} \label{c:u-V_2D}
For all $\veps\in\,]0,1]$, set $V_\veps\,=\,\rho_\veps\,u_\veps$. Let $0<\k<1$ and $0<\wtilde s<1$ be the indices defined in Proposition \ref{p:s-uniform}. Then
$$
V_{\veps}\,=\,u_{\veps}\,+\,\veps^\k\,\mc V_{\veps}\,,
$$
where the sequence $\big(\mc V_{\veps}\big)_\veps\,\subset\,L^2_T(H^{-\wtilde s})$ for all $T>0$ fixed.

Moreover, let us define $\eta_\veps\,:=\,\curl V_\veps$ and $\omega_\veps\,:=\,\curl u_\veps$. For every $\veps\in\,]0,1]$, one has
the following properties:
\begin{align*}
&\eta_{\veps}\,=\,\omega_{\veps}\,+\,\veps^\k\,\z_{\veps}\qquad\qquad \mbox{ and }\qquad\qquad
\div V_{\veps}\,=\,\veps^\beta\,\theta_\veps\,+\,\veps^\k\,d_\veps\,,
\end{align*}
where both sequences $\big(\z_\veps\big)_\veps$ and $\big(d_\veps\big)_\veps$ are uniformly bounded in $L^2_T(H^{-\wtilde s-1})$, for all $T>0$ fixed.
\end{coroll}

Another consequence of Proposition \ref{p:s-uniform} is that it allows to reduce the nonlinearity of the convective term. More precisely, the following statement, analogous to Lemma \ref{l:1_conv},
holds true.
\begin{lemma} \label{l:1_conv_0}
For any test function $\psi\,\in\,\mc{D}\bigl([0,T[\,\times\R^2;\R^2\bigr)$, one has
$$
\lim_{\veps\ra0^+}\left|\int^T_0\int_{\R^2}\rho_\veps u_\veps\otimes u_\veps:\nabla\psi\,dx\,dt\,-\,\int^T_0\int_{\R^2} u_\veps\otimes u_\veps:\nabla\psi\,dx\,dt\right|\,=\,0\,.
$$
\end{lemma}

\begin{proof}
The proof simply relies on uniform bounds and continuity properties of the product in Sobolev spaces. First of all, we decompose $\rho_\veps$ according to
$$
\rho_\veps\,=\,1\,+\,\veps^\k\left(\frac{1}{\veps^\k}\,r_\veps\right)\,.
$$
In view of \eqref{ub:u-Du} and \eqref{est:u-L^2}, point (iii) of Corollary \ref{c:product-2} implies that $\big(u_\veps\otimes u_\veps\big)_\veps$ is uniformly bounded in
$L^1_T(H^{1-\de})$, for any $T>0$ fixed and any $\de>0$ small enough. On the other hand, thanks to Proposition \ref{p:s-uniform} we know that
$\big(\veps^{-\k}\,r_\veps\big)_\veps$ is uniformly bounded in $L^\infty_T(H^{-\wtilde s})$, with $\wtilde s<1$. Therefore, taking $\de>0$ small enough, we deduce
from point (i) of Corollary \ref{c:product-2} that
$$
\left\|\frac{1}{\veps^\k}\,r_\veps\,u_\veps\otimes u_\veps\right\|_{L^1_T(H^{-\wtilde s-\de})}\,\leq\,C(T)\,,
$$
for a positive constant $C(T)$ depending only on the fixed time $T$. From this uniform bound, the result easily follows.
\end{proof}

After those preliminaries, the rest of the convergence proof is pretty much similar to the previous one, for $\alpha>0$. Let us draw it for the reader's convenience.

\subsection{Convergence in the weak formulation of the equations} \label{ss:conv_0}

In this subsection, we complete the proof to Theorem \ref{t:0}: namely, we pass to the limit in the weak formulation of our equations. Recall that, throughout
this part, the equations are set in $\R_+\times\R^2$, and we have assumed $\alpha=0$ and $\g>1$ in \eqref{hyp:P}.

Also in this case, we have to pass to the limit in the momentum equation only.
Indeed, we have already shown in Proposition \ref{p:constr_0} that the mass equation simply vanishes in the limit; the reason is that we have not enough regularity
on the functions $\s_\veps$ to infer an equation for them.

Therefore, given a test function $\psi\,=\,\nabla^\perp\vphi$, where $\vphi\in\mc D\big([0,T[\,\times\R^2\big)$ for some $T>0$, let us consider the equality
\begin{align} \label{eq:weak-0}
\int^T_0\!\!\!\int_{\R^2}\left(-\rho_\veps u_\veps\cdot\d_t\psi-\rho_\veps u_\veps\otimes u_\veps:\nabla\psi+\dfrac{1}{\veps}\rho_\veps u^\perp_\veps\cdot\psi+
\mu\nabla u_\veps:\nabla\psi\right)=
\int_{\R^2}\rho_{0,\veps}u_{0,\veps}\cdot\psi(0)\,.
\end{align}

It goes without saying that passing to the limit on the initial datum and viscosity terms present no difficulty, and can be done as in Subsection \ref{ss:first-conv}.
The convergence of the time derivative term is also easy, once one uses e.g. Proposition \ref{p:s-uniform}.
In addition, the Coriolis term can be treated exactly as in the $3$-D case: we obtain
\begin{align*}
\frac{1}{\veps}\int^T_0\!\!\!\int_{\R^2}\rho_\veps\,u_\veps^\perp\cdot\nabla^\perp\vphi\,&=\,-\int_0^T\!\!\!\int_{\R^2}\s_\veps\,\d_t\vphi\,-\,\int_{\R^2}r_{0,\veps}\,\vphi\;
\longrightarrow\;-\int_0^T\!\!\!\int_{\R^2}\s\,\d_t\vphi\,-\,\int_{\R^2}r_{0}\,\vphi\,,
\end{align*}
where we recall that $r_0$ has been introduced in \eqref{cv:in-data}.

Therefore, in order to complete the passage to the limit, and then the proof of Theorem \ref{t:0}, it remains us to prove the convergence of the convective term
$\rho_\veps\,u_\veps\otimes u_\veps$. We observe that, in view of Lemma \ref{l:1_conv_0} above, it is enough to pass to the limit in the integral
$$
-\,\int^T_0\int_{\R^2} u_\veps\otimes u_\veps:\nabla\psi\,dx\,dt\,,\qquad\qquad\mbox{ with }\qquad \psi\,=\,\nabla^\perp\vphi\,.
$$
Our method to prove convergence is based on compensated compactness, analogously to what done for $0<\alpha<1$; actually, the argument is simpler here, because
we are in two space dimensions, hence we have to deal only with a term similar to $\mc T^1_{\veps,M}$ above.

Let us sketch the argument. Omitting a standard regularisation process, we can suppose the velocity field $u_\veps$, and all the other quantities, to be smooth with respect to the space variables;
on the contrary, concerning time integrability of the different quantities, we cannot hope for anything better than what is given by the uniform bounds established before.
Owing to the regularity in space, we can inegrate by parts: we find
\begin{align*}
-\,\int^T_0\int_{\R^2} u_\veps\otimes u_\veps:\nabla\psi\,&=\,\int^T_0\int_{\R^2}\div\big(u_\veps\otimes u_\veps\big)\cdot\nabla^\perp\vphi \\
&=\,\int^T_0\int_{\R^2}\left(\div u_\veps\,u_\veps\,+\,\frac{1}{2}\,\nabla\left|u_\veps\right|^2\,+\,\omega_\veps\,u_\veps^\perp\right)\cdot\nabla^\perp\vphi\,,
\end{align*}
where we have resorted to the same computations of \eqref{eq:T^1}, adapted to the $2$-D case. Since $\div u_\veps$ is of order $\veps^\beta$ and the test function is divergence-free,
the first two terms in the right-hand side converge to $0$ in the limit when $\veps\ra0^+$. So, it remains us to pass to the limit in
the bilinear term $\omega_\veps\,u_\veps^\perp$.

For this, first of all we observe that, owing to Corollary \ref{c:u-V_2D}, it is enough to consider the product $\eta_\veps\,u^\perp_\veps$:
at this point, we need some strong convergence property. Working with the regularity index $s_0$ defined in \eqref{def:s} where we take $\alpha=0$, we can apply here the same agument of
Paragraph \ref{sss:strong-vort} to infer the following property: for all $\de>0$, there exists a $\veps_\de$ such that
$$
\big(\eta_\veps\big)_{\veps\leq\veps_\de}\qquad\qquad\mbox{ is compact in }\qquad L^2\big(\,]\de,T[\,;L^2_{\rm loc}(\R^2)\big)\,.
$$
Therefore, on the one hand we gather the convergence
$$
\int^T_\de\int_{\R^2}\eta_\veps\,u_\veps^\perp\cdot\nabla^\perp\vphi\,\longrightarrow\,\int^T_\de\int_{\R^2}\omega\,u^\perp\cdot\nabla^\perp\vphi
$$
when $\veps\ra0^+$; on the other hand, by uniform bounds one has
$$
\left|\int^\de_0\int_{\R^2}\eta_\veps\,\,u^\perp_\veps\cdot\nabla^\perp\vphi\right|\,\leq\,C_\de\,,
$$
with $C_\de\longrightarrow0$ for $\de\ra0^+$. Therefore, by arbitrariness of $\de>0$, we finally deduce the convergence
$$
-\,\int^T_0\int_{\R^2} u_\veps\otimes u_\veps:\nabla\nabla^\perp\vphi\,dx\,dt\,\longrightarrow\int^T_0\int_{\R^2} \omega\,u^\perp_\veps\cdot\nabla^\perp\vphi\,dx\,dt\,=\,
\int^T_0\int_{\R^2} \omega\,u_\veps\cdot\nabla\vphi\,dx\,dt\,.
$$

The previous convergence property completes the proof of Theorem \ref{t:0}.

\subsection{A conditional convergence result} \label{ss:dens-full}

In this subsection, we state and prove a convergence result for the fully non-homogeneous case, where we are able to pass to the limit to the full system,
in which the dynamics of the density fluctuation function and the velocity field are decoupled.

This is just a conditional result, because  very strong assumptions are required on the sequence of weak solutions: in particular, we need to assume uniform bounds in higher norms for the sequence
of velocity fields (see in particular conditions (ii)-(iii)-(iv) in Theorem \ref{t:2D-full} below). Those bounds cannot be deduced from classical energy estimates, while higher order
energy estimates seem to be not uniform in the singular parameter $\veps$.

The statement is the following one. Recall that $\Omega\,=\,\R^2$ in the case $\alpha=0$.
\begin{thm} \label{t:2D-full}
With the notation and under the assumptions of Theorem \ref{t:0}, assume moreove that the following conditions hold true:
\begin{enumerate}[(i)]
 \item $\bigl(r_{0,\veps}\bigr)_\veps\,\subset\,H^{\mf{b}}(\Omega)$, for some $\mf{b}\in\,]0,1]$;
 \item $\bigl(u_\veps\bigr)_\veps\,\subset\,L^{\infty}_{\rm loc}\bigl(\R_+;H^{1}(\Omega)\bigr)\cap L^{2}_{\rm loc}\bigl(\R_+;H^{2}(\Omega)\bigr)$;
 \item $\bigl(u_\veps\bigr)_\veps\,\subset\,\mc C^{0,\mf a}_{\rm loc}\bigl(\R_+;L^{2}(\Omega)\bigr)$, for some $\mf{a}\in\,]0,1[\,$;
 \item $\bigl(\div u_\veps\bigr)_\veps\,\subset\,L^{1}_{\rm loc}\bigl(\R_+;L^\infty(\Omega)\bigr)$.
\end{enumerate}
Let $\theta$ be the quantity introduced in \eqref{cv:div-u}. Let $r_0$ and $u_0$ be the functions defined in \eqref{cv:in-data}.
Finally, let $\delta_{1\beta}$ the ``Kronecker delta'', namely $\delta_{1\beta}=1$ if $\beta=1$ and $\delta_{1\beta}=0$ otherwise, where $\beta\geq1$ is the parameter
appearing in \eqref{eq:sing-NSC_2D}.

Then there exist a distribution $\Pi$ over $\R_+\times\Omega$ such that the limit points $\big(\sigma,u,\theta)$ satisfy the system
\begin{equation} \label{eq:dens-full}
\left\{\begin{array}{l}
        \d_t\s\,+\,u\cdot\nabla\s\,+\,\delta_{1\beta}\,\theta\,=\,0 \\[1ex]
        \d_tu\,+\,u\cdot\nabla u\,+\,\nabla\Pi\,+\,\s\,u^\perp\,-\,\de_{1\beta}\,\nabla^\perp(-\Delta)^{-1}\theta\,-\,\mu\,\Delta u\,=\,0 \\[1ex]
         \div u\,=\,0\,,
       \end{array}
\right. 
\end{equation}
with initial data $\s_{|t=0}\,=\,r_0$ and $u_{|t=0}\,=\,u_0$.
\end{thm}

It goes without saying that, under the previous assumptions, the convergence properties for $\bigl(u_\veps\bigr)_\veps$ and $\bigl(\s_\veps\bigr)_\veps$
stated in Theorem \ref{t:0} can be improved (see also Proposition \ref{p:sigma_e} below).
However, for simplicity we refrain from doing that: our focus here is on obtaining convergence to the full system rather than \eqref{eq:limit-0}.

We also remark that, according to Proposition \ref{p:constr_0}, $u\,=\,-\nabla^\perp\theta$, hence $\theta\,=\,(-\Delta)^{-1}\omega$, where $\omega=\curl u$. Therefore, system \eqref{eq:dens-full}
is in fact a system for the couple of unknowns $(\s,u)$.

The rest of this section is devoted to the proof of the previous result.

\subsubsection{Regularity of the density oscillations} \label{sss:sigma-reg}

The first step is to gain space regularity for the density oscillation functions $\sigma_\veps$. This is possible thanks to the additional assumptions (i), (ii) and (iv) in Theorem \ref{t:2D-full}.
Thanks to that property, and using also assumption (iii), we will then derive compactness for $\sigma_\veps$ in space-time.

The preliminary remark is that the only way to recover any information for $\sigma_\veps$ is to use \eqref{eq:curlV-sigma}, which in dimension $d=2$ becomes
\begin{equation} \label{eq:eta-s_2D}
\d_t\big(\eta_\veps\,-\,\s_\veps\big)\,=\,\curl f_\veps\,.
\end{equation}
Recall that we have set $\eta_\veps\,:=\,\curl V_\veps$. Hence, in order to improve the space regularity of $\sigma_\veps$, we need to improve the regularity of $\eta_\veps$ and $f_\veps$,
which in turn forces us to seek for additional smoothness of the functions $r_\veps$. Propagating this last property is possible thanks to a slight adaptation of Theorem 3.33 of \cite{B-C-D} (combined also
with Remark 3.35 therein), which deals with transport equations with a velocity field which is almost Lipschitz (see also \cite{D_2005} for further results in this direction).

\begin{lemma} \label{l:transport}
Let $-1< s<2$ and $T>0$ fixed. Let $u\in L^1_{T}(H^2)$ such that $\div u\in L^1_T(L^\infty)$, $r_0\in B^s_{2,\infty}$ and $g\in \wtilde{L}^1_{T}(B^s_{2,\infty})$ be given.
Then the continuity equation
$$
\d_tr\,+\,u\cdot\nabla r\,+\,r\,\div u\,=\,g
$$
admits a unique solution $r\in\mc C\big([0,T];\bigcap_{s'<s}B^{s'}_{2,\infty}\big)$, and the following estimates hold true, for all $\de>0$ arbitrarily small:
$$
\left\|r\right\|_{\wtilde{L}^\infty_T(B^{s-\de}_{2,\infty})}\,\leq\,C\,\exp\left(\frac{C}{\de}\left(\int^T_0\big(\left\|\nabla u\right\|_{H^1}\,+\,
\left\|\div u\right\|_{L^\infty}\big)\,dt\right)^{\!\!2}\right)\,\Big(\left\|r_0\right\|_{B^s_{2,\infty}}\,+\,\left\|g\right\|_{\wtilde{L}^1_T(B^{s}_{2,\infty})}\Big)\,.
$$
The constant $C$ only depends on $s$.
\end{lemma}

\begin{proof}
It is enough to apply Theorem 3.33 of \cite{B-C-D} with $p=p_1=2$, $\alpha=1/2$ and $d=2$. The only thing which needs some verification is the fact that Remark 3.35 applies
to $r\,\div u$, which has to be treated as a forcing term.

Thanks to the \textsl{a priori} bound $r\in B^{s'}_{2,\infty}$ for all $s'<s$ and to the property $\div u\in H^1$, the product rules of Proposition \ref{p:op} immediately imply
that the product $r\,\div u$ belongs to $H^{s'}$ (since we are in dimension $d=2$), which is included in $B^{s'}_{2,\infty}$. Therefore
$$
\left\|\Delta_j\big(r\,\div u\big)(t)\right\|_{L^p}\,\leq\,C\,2^{-js'}\,\left\|r(t)\right\|_{B^{s'}_{2,\infty}}\,\left\|\div u(t)\right\|_{H^1\cap L^\infty}
$$
for all $j\geq -1$, for all $s'\,\in\,]s-\de,s[$ and all $t\in[0,T]$, where $\left\|\div u(t)\right\|_{H^1\cap L^\infty}$ is integrable over $[0,T]$.
The last inequality completes the proof of the proof of the lemma.
\end{proof}

From the previous statement, we can derive additional regularity properties for the density variations $r_\veps$.
\begin{coroll} \label{c:reg-r}
Under hypotheses (i) and (ii) of Theorem \ref{t:2D-full}, one has
$$
\big(r_\veps\big)_\veps\,\subset\,\mc C_{\rm loc}\big(\R_+;H^{\mf b'}\big)\,\cap\,\mc C^{0,1/2}_{\rm loc}\big(\R_+;H^{\mf b'-1}\big)
$$
for all $0\leq\mf b'<\mf b$. In particular, $\big(r_\veps\big)_\veps$ is compact in the space $\mc C\big([0,T];L^2_{\rm loc}\big)$ for all $T>0$ fixed.
\end{coroll}

\begin{proof}
Fix $\veps\in\,]0,1]$. By definition, $r_\veps\,=\,\rho_\veps-1$ verifies the continuity equation
$$
\d_tr_\veps\,+\,u_\veps\cdot\nabla r_\veps\,+\,r_\veps\,\div u_\veps\,=\,-\,\div u_\veps\,,
$$
related to the initial datum $\big(r_\veps\big)_{|t=0}\,=\,\veps\,r_{0,\veps}$. By hypothesis (i), (ii) and (iv) of Theorem \ref{t:2D-full} we deduce respectively that
$r_{0,\veps}\in H^{\mf b}$, with $0<\mf b\leq 1$, and $\div u_\veps\in L^2_T(H^1)$ for all $T\geq0$. Hence, a straightforward application of Lemma \ref{l:transport} implies
that $r_\veps\in \mc C\big([0,T];H^{\mf b'}\big)$ for all $0\leq\mf b'<\mf b$. Moreover, the estimate given in the same lemma above yields that the whole sequence 
$\big(r_\veps\big)_\veps$ is uniformly bounded in the previous space.

Next, let us write
$$
\d_tr_\veps\,=\,-\,\div\big(r_\veps\, u_\veps\big)\,-\,\div u_\veps\,,
$$
where $\big(\div u_\veps\big)_\veps$ is uniformly bouned in $L^2_T(H^1)$ for all $T>0$. In addition, by the embedding $H^2\hookrightarrow L^\infty$, using the previous
uniform bounds for $\big(r_\veps\big)_\veps$, assumption (ii) of Theorem \ref{t:2D-full} and the product rules (iii) of Corollary \ref{c:product-2}, one gathers that
$\big(r_\veps\,u_\veps\big)_\veps$ is uniformly bounded in $L^2_T(H^{\mf b'})$ for all $\mf b'<\mf b$. From those properties we derive that
$\big(\d_tr_\veps\big)_\veps\,\subset\,L^2_T(H^{\mf b'-1})$, from which the uniform embedding $\big(r_\veps\big)_\veps\,\subset\,\mc C^{0,1/2}_T(H^{\mf b'-1})$ easily follows.

Ascoli-Arzel\`a theorem and an interpolation with the previous uniform bounds immediately give also the compactness property.
\end{proof}

We are now ready to derive better space regularity for the functions $\sigma_\veps$.
\begin{prop} \label{p:sigma_e}
Let assumptions (i) and (ii) of Theorem \ref{t:2D-full} hold. Then, for all $0<\mf b'<\mf b$, one has
$$
\big(\s_\veps\big)_\veps\,\subset\,L^\infty_{\rm loc}\big(\R_+;H^{\mf b'-2}(\Omega)\big)\,.
$$
In particular, $\big(\s_\veps\,u_\veps\big)_\veps$ is uniformly bounded in $L^2_{\rm loc}\big(\R_+;H^{\mf b'-2}(\Omega)\big)$.

In addition, under assumption (iii) of Theorem \ref{t:2D-full}, the sequence $\big(\s_\veps\big)_\veps$ is compact in the space $L^\infty_T(H^{\mf b'-2}_{\rm loc})$, for all times $T>0$ and all
indices $0<\mf b'<\mf b$. In particular, one gathers the weak convergence
$$
\sigma_\veps\,u_\veps\,\rightharpoonup\,\s\,u\quad \mbox{ in }\quad L^2_{\rm loc}\big(\R_+;H_{\rm loc}^{\mf b'-2}\big)\,.
$$
\end{prop}

\begin{proof}
Let us consider equation \eqref{eq:eta-s_2D}: an integration in time yields, for almost every $t\geq0$ and for all $\veps\in\,]0,1]$, the relation
$$
\sigma_\veps(t)\,=\,\eta_\veps\,-\,\curl\big(\rho_{0,\veps}\,u_{0,\veps}\big)\,+\,r_{0,\veps}\,-\,\int^t_0\curl f_\veps(\tau)\,d\tau\,.
$$
Recall that $f_\veps$ has been defined in Subsection \ref{ss:further}.

By assumption, $\big(r_{0,\veps}\big)_\veps\,\subset\, H^{\mf b}$, for some $0<\mf b\leq1$, while the family of $\curl\big(\rho_{0,\veps}\,u_{0,\veps}\big)$'s is uniformly bounded
in $H^{-1}$. In addidion, in view of the uniform bounds $\big(r_\veps\big)_\veps\subset L^\infty_T(H^{\mf b'})$ for all $0\leq \mf b'<\mf b$ and $\big(u_\veps\big)_\veps\subset L^\infty_T(H^1)$,
item (ii) of Corollary \ref{c:product-2} implies that $\big(\eta_\veps\big)_\veps\,\subset\,L^\infty_T(H^{\mf b'-1})$ for all $0\leq \mf b'<\mf b$
and all $T>0$.

Next, we remark that the $\curl$ operator kills the gradient of the pressure appearing in the definition of $f_\veps$. Then we get
\begin{align*}
\curl f_\veps\, &=\,\mu\,\Delta\omega_\veps\,-\,\curl\div\big(\rho_\veps\, u_\veps\otimes u_\veps\big) \\
&=\,\mu\,\Delta\omega_\veps\,-\,\curl\div\big(u_\veps\otimes u_\veps\big)\,-\,\curl\div\big(r_\veps\, u_\veps\otimes u_\veps\big)\,,
\end{align*}
where $\omega_\veps\,=\,\curl u_\veps$ as usual. Now, $\big(\Delta\omega_\veps\big)_\veps\subset L^2_T(H^{-1})$ in view of assumption (ii) of Theorem \ref{t:2D-full}; moreover, from item (iii)
of Corollary \ref{c:product-2} we infer that $\big(u_\veps\otimes u_\veps\big)_\veps\subset L^2_T(H^1)$. Hence, on the one hand
$\Big(\curl\div\big(u_\veps\otimes u_\veps\big)\Big)_\veps$ is uniformly bounded in $L^2_T(H^{-1})$; on the other hand, using also item (ii) of Corollary \ref{c:product-2}, we get that
$\Big(\curl\div\big(r_\veps\,u_\veps\otimes u_\veps\big)\Big)_\veps$ is uniformly bounded in $L^2_T(H^{\mf b'-2})$ for all $0\leq\mf b'<\mf b$.

Putting all those properties together, we finally deduce that $\big(\s_\veps\big)_\veps$ is uniformly bounded in $L^\infty_T(H^{\mf b'-2})$ for all $T>0$ and all $0\leq \mf b'<\mf b$.
By item (iii) of Corollary \ref{c:product-2} one also gathers the uniform boundedness of $\big(\s_\veps\,u_\veps\big)_\veps$ in $L^2_T(H^{\mf b'-2})$.

Next, we remark that, from \eqref{eq:eta-s_2D} and the previous analysis of $\curl f_\veps$, we infer that the sequence $\big(\eta_\veps-\s_\veps\big)_\veps$ is compact in e.g. the space
$\mc C\big([0,T];H^{\mf b'-2}_{\rm loc}\big)$, for all $\mf b'<\mf b$. Now, using Proposition \ref{p:s-uniform} and arguing exactly as in its proof, we can decompose
$$
\eta_\veps\,=\,\omega_\veps\,+\,\veps^\kappa\,\curl\left(\veps^{-\k}\,r_\veps\,u_\veps\right)\,,
$$
where, thanks to the fact that $\big(u_\veps\big)_\veps\subset L^\infty_T(H^1)$, we have that
$\Big(\curl\left(\veps^{-\k}\,r_\veps\,u_\veps\right)\Big)$ is uniformly bounded in $L^\infty_T(H^{-\wtilde s-\de})$ for all $\de>0$ arbitrarily small, with $\wtilde s\in\,]0,1[\,$ fixed.
Finally, in view of assumption (iii) of Theorem \ref{t:2D-full}, we get that $\big(\omega_\veps\big)_\veps$ is compact in $\mc C_T(H^{-1-\de}_{\rm loc})$ for all $\de>0$ small.

All these properties together immediately imply the compactness of $\big(\s_\veps\big)_\veps$ in $L^\infty_T(H^{\mf b'-2}_{\rm loc})$.
Now, combining this latter strong convergence with the uniform bound $\big(u_\veps\big)_\veps\subset L^2_T(H^2)$ and the product rules stated in item (iii) of Corollary \ref{c:product-2},
we also deduce the convergence $\s_\veps\,u_\veps\,\rightharpoonup\,\s\,u$ in $L^2_T(H^{\mf b'-2}_{\rm loc})$.

The proof of the proposition is now completed.
\end{proof}

\subsubsection{The proof of the convergence}

We are now in the position of showing convergence in system \eqref{eq:sing-NSC_2D}, completing in this way the proof to Theorem \ref{t:2D-full}.

First of all, we rewrite the mass equation in the following form:
$$
\d_t\s_\veps\,+\,\div\big(\s_\veps\,u_\veps\big)\,+\,\veps^{\beta-1}\,\theta_\veps\,=\,0\,.
$$
At this point, it is easy to pass to the limit in the previous equation, in view of Proposition \ref{p:sigma_e} above: as claimed, we get the relation
$$
\d_t\s\,+\,\div\big(\s\,u\big)\,+\,\de_{1\beta}\,\theta\,=\,0\,.
$$

On the other hand, there is no more need of passing to the vorticity formulation when proving the weak convergence of the momentum equation.
Observe that the time derivative, the viscosity term and the convective term can be dealt with as in Subsection \ref{ss:conv_0}, while the gradient terms disappear, because the test function
$\psi\,=\,\nabla^\perp\vphi$ is divergence-free.

Finally, it remains us to pass to the limit in the Coriolis term, for which we can argue in the following way:
\begin{align*}
\frac{1}{\veps}\int^T_0\!\!\!\int_{\R^2}\rho_\veps\,u_\veps^\perp\cdot\psi\,&=\,\frac{1}{\veps}\int^T_0\!\!\!\int_{\R^2}u_\veps^\perp\cdot\psi\,+\,
\int^T_0\!\!\!\int_{\R^2}\s_\veps\,u_\veps^\perp\cdot\psi\,.
\end{align*}
The latter term in the right-hand side of the previous equality obviously converges, thanks to Proposition \ref{p:sigma_e}. As for the former term, instead, we can use the fact that
$\psi\,=\,\nabla^\perp\vphi$ to get
$$
\frac{1}{\veps}\int^T_0\!\!\!\int_{\R^2}u_\veps^\perp\cdot\psi\,=\,\frac{1}{\veps}\int^T_0\!\!\!\int_{\R^2}u_\veps\cdot\nabla\vphi\,=\,
-\,\veps^{\beta-1}\int^T_0\!\!\!\int_{\R^2}\theta_\veps\,\vphi\,\longrightarrow\,-\,\de_{1\beta}\int^T_0\!\!\!\int_{\R^2}\theta\,\vphi\,.
$$
Now, we use the fact that $\vphi\,=\,-(-\Delta)^{-1}\curl\psi$ to write (when $\beta=1$)
$$
-\int^T_0\!\!\!\int_{\R^2}\theta\,\vphi\,=\,\int^T_0\!\!\!\int_{\R^2}\theta\,(-\Delta)^{-1}\curl\psi\,=\,-\int^T_0\!\!\!\int_{\R^2}\nabla^\perp(-\Delta)^{-1}\theta\cdot\psi\,.
$$
The proof to Theorem \ref{t:2D-full} is hence completed.


\end{document}